\documentclass[pdflatex,sn-mathphys-num]{sn-jnl}

\usepackage{geometry}
\usepackage{graphicx}%
\usepackage{multirow}%
\usepackage{amsmath,amssymb,amsfonts}%
\usepackage{amsthm}%
\usepackage{mathrsfs}%
\usepackage[title]{appendix}%
\usepackage{xcolor}%
\usepackage{textcomp}%
\usepackage{manyfoot}%
\usepackage{booktabs}%
\usepackage{algorithm}%
\usepackage{algorithmicx}%
\usepackage{algpseudocode}%
\usepackage{listings}%
\usepackage{comment}
\usepackage{subcaption}
\usepackage{hyperref}
\usepackage{siunitx}
\usepackage{multirow}

\usepackage{pgfplots}
\pgfplotsset{compat=1.18}
\usepgfplotslibrary{fillbetween}
\usetikzlibrary{calc}
\usepackage{cancel}
\usepackage{tikz}
\usepackage{tikz-3dplot}

\theoremstyle{thmstyleone}%
\newtheorem{theorem}{Theorem}
\newtheorem{proposition}[theorem]{Proposition}%

\theoremstyle{thmstyletwo}%
\newtheorem{remark}{Remark}%

\theoremstyle{thmstylethree}%
\newtheorem{definition}{Definition}%

\DeclareMathOperator*{\argmin}{\arg\min}
\DeclareMathOperator*{\argmax}{\arg\max}

\begin{document}

\title[Nonconvex optimization methods for ground states in disordered continuous-spin models]{Nonconvex optimization methods for ground states in disordered continuous-spin models}

\author[1]{\fnm{Ramgopal} \sur{Agrawal}}\email{ramgopal.agrawal@uniroma1.it}

\author[2]{\fnm{Lorenzo} \sur{Ciarpaglini}}\email{lorenzo.ciarpaglini@uniroma1.it}

\author[3]{\fnm{Pierluigi} \sur{Mansueto}}\email{ pierluigi.mansueto@unifi.it}

\author[1]{\fnm{Enzo} \sur{Marinari}}\email{enzo.marinari@uniroma1.it}

\author[2]{\fnm{Marco} \sur{Sciandrone}}\email{marco.sciandrone@uniroma1.it}

\author[2]{\fnm{Diego} \sur{Scuppa}}\email{diego.scuppa@uniroma1.it}

\author*[2]{\fnm{Elisa} \sur{Trasatti}}\email{elisa.trasatti@uniroma1.it}

\affil[1]{\orgdiv{Department of Physics}, \orgname{Sapienza University of Rome}, \orgaddress{\street{Piazzale Aldo Moro 2}, \city{Roma}, \postcode{00185},  \country{Italy}}}

\affil[2]{\orgdiv{Department of Computer, Control and Management Engineering}, \orgname{Sapienza University of Rome}, \orgaddress{\street{Via Ariosto 25}, \city{Roma}, \postcode{00185},  \country{Italy}}}

\affil[3]{\orgdiv{Department of Information Engineering}, \orgname{University of Florence}, \orgaddress{\street{Via S. Marta 3}, \city{Firenze}, \postcode{501395},  \country{Italy}}}

\abstract{\unboldmath 
This work explores the global optimization problem of finding lowest-energy configurations in disordered continuous-spin models from statistical physics, with a particular focus on the random field $XY$ model. Due to an extremely non-convex nature of the associated energy landscape, this problem remains highly challenging. From an optimization perspective, we reformulate the traditional angular Hamiltonian as a constrained problem on the Cartesian product of spheres, allowing the application of Riemannian optimization techniques, which show better computational performance. We design a family of Basin Hopping algorithms whose perturbation mechanisms are specifically designed to exploit the structure of the underlying physical model, and further extend them within a Population Basin Hopping framework. The proposed methods are evaluated against optimization algorithms widely used in computational physics. The proposed  variants turn out to be the most effective method in the comparison, consistently attaining lower-energy configurations within the same computational budget. This work establishes a robust link between continuous-spin systems and continuous global optimization, providing a high-performance benchmark for exploring complex energy landscapes.}



\maketitle

\newpage

\section{Introduction}\label{sec:intro}

Global optimization problems are central to a wide range of statistical physics problems~\cite{Wales_2004,10.5555/1592967}. These problems typically arise when the underlying energy landscape is highly complex and characterized by a large number of local optima. While discrete energy minimization problems can often be addressed using a variety of combinatorial methods such as graph-cut methods~\cite{hartmann2006phase,969114,10.1145/227683.227684}, Tree-Reweighted Message Passing~\cite{1677515}, semidefinite programming~\cite{10.1145/3005345}, the extension to continuous variables leads to significantly more challenging global optimization problems, particularly in the presence of non-convex objective (energy) functions.
\par\medskip\noindent
In this work, we consider disordered continuous-spin systems from statistical physics, whose energy landscapes become highly rugged due to the interplay between continuous spin variables and quenched disorder. Representative examples include spin-glass models with frustrated interactions~\cite{PhysRevB.78.014419,PhysRevE.94.052143,PhysRevB.91.134203,Baity-Jesi_2019} and random field models~\cite{PhysRevB.53.15193,PhysRevB.88.224418,Lupo_2019,XYmodel}. Determining the lowest-energy configurations (ground states) of these systems constitutes a difficult large-scale global optimization problem. The computational difficulty originates from the existence of an extremely large number of metastable states separated by large energy barriers, which severely limits the effectiveness of standard optimization procedures such as steepest descent and simulated annealing~\cite{doi:10.1126/science.220.4598.671,PhysRevB.88.224418,PhysRevB.22.3816}. A few numerical approaches have been proposed to improve the computation of ground states in continuous-spin systems. In two-dimensional spin-glass models, algorithms inspired by minimum-weight perfect matching exploit the topology of domain-wall excitations to construct collective global updates capable of overcoming large barriers~\cite{PhysRevLett.96.097206,PhysRevE.76.066706}. However, these methods are specifically tailored to two-dimensional spin-glass systems and cannot be readily extended to higher dimensions or to random field models. Another class of approaches relies on discretizing the continuous spin variables~\cite{Lupo_2019}, thereby enabling the application of combinatorial optimization techniques. Nevertheless, the resulting optimization problem remains \textit{NP}-hard~\cite{hartmann2006phase,969114,PhysRevE.97.053307}, while the quality of the computed ground state depends on the discretization level. Consequently, efficient optimization frameworks capable of treating the original continuous problem directly remain highly desirable.
\par\medskip\noindent
The present work adopts a different perspective by treating ground-state computation as a continuous global optimization problem. As a representative case study, we consider the three-dimensional random field $XY$ model. This model has been extensively investigated in statistical physics due to its rich disorder-induced features and the absence of a conventional ordered phase~\cite{PhysRevB.53.15193,PhysRevB.88.224418,Lupo_2019,XYmodel}. From an optimization viewpoint, it constitutes an excellent benchmark because the presence of on-site random field disorder gives rise to an extremely rugged energy landscape with numerous metastable states.

\par\medskip\noindent
We observe that the problem can be formulated in two equivalent ways (see details of the objective function in Section~\ref{sec:problem}): 
\begin{enumerate}
	\item[a)] as an unconstrained optimization problem in a Euclidean space; 
	\item[b)] as an optimization problem on a Riemannian manifold. 
\end{enumerate}
The first formulation is based on angular variables and trigonometric functions and will be exploited to derive theoretical insights from a global optimization perspective. The second formulation is a constrained quadratic problem with a feasible set given by the Cartesian product of spheres and will be adopted in the numerical experiments. The motivation is primarily computational: the objective function, its gradient, and its Hessian involve only inner products between two-component vectors, rather than trigonometric functions, so that their evaluation is considerably cheaper (see details in Section~\ref{sec:computational efficiency}). Moreover, various efficient and effective solvers are available for optimization on manifolds. Building on this geometric formulation, we design a family of Basin Hopping algorithms whose perturbation mechanisms are tailored to the structure of the feasible set and the objective function. We further extend this framework to a Population Basin Hopping scheme, in which a set of candidate configurations evolves collectively to improve exploration of the energy landscape. We compare the resulting algorithms against a MultiStart strategy, the genetic algorithm Differential Evolution \cite{Storn1997}, and optimization methods traditionally employed in the physics literature for this class of problems, including Parallel Tempering \cite{Hukushima1996,marinari1992} and Simulated Annealing \cite{doi:10.1126/science.220.4598.671}, showing that the proposed variants consistently attain the most competitive results among all the strategies considered. In this context, the main contributions of the present work are as follows:
\begin{enumerate}
\item [1)] the design and implementation of a global optimization framework  useful for helping to better understand the underlying statistical physics problem;
\item [2)] the definition of a benchmark for a class of global optimization problems with putative global optima.
\end{enumerate}

\par\noindent
Such benchmark instances are of independent interest, as they provide structured yet challenging test cases for assessing global optimization algorithms on non-convex problems. 
Finally, we expect that this work will foster stronger interaction between the statistical physics and global optimization communities.

\par\medskip\noindent
The paper is organized as follows. Section 2 describes and defines the unconstrained optimization problem and presents a study of the properties of the objective function. Section 3 derives an equivalent formulation in terms of optimization on a Riemannian manifold. Section 4 presents the proposed Basin Hopping algorithms with problem-specific perturbation strategies and their population-based counterparts. Section 5 presents the experimental results, including a comparison of the two formulations, an assessment of the proposed perturbation strategies, and a benchmark against the MultiStart scheme and established optimization methods from the computational physics literature. Finally, Section 6 contains the conclusions and directions for future work. 


\section{Formulation of the unconstrained optimization problem}\label{sec:problem}
 The random field $XY$ model is defined on a $d$-dimensional lattice of size $L$ with periodic boundary conditions (see Figure \ref{Lattice}). Each site $i$ of the lattice hosts a two-component unit vector spin of orientation $\theta_i$ and a two-component unit vector random field of orientation $\phi_i$.

\begin{figure}[ht]
\centering
\begin{tikzpicture}[scale=1.2]
\def\L{4}
\foreach \x in {0,...,3}
  \foreach \y in {0,...,3}
{
    \fill[black] (\x,\y) circle (0.06);

    \ifnum\x<3
      \draw[gray, thin] (\x,\y) -- ({\x+1},\y);
    \fi
    \ifnum\y<3
      \draw[gray, thin] (\x,\y) -- (\x,{\y+1});
    \fi
}
\foreach \y in {0,...,3}
{
  \draw[gray, dashed, thin]
    (3,\y) .. controls +(0.6,0.3) and +(-0.6,0.3) .. (0,\y);
}
\foreach \x in {0,...,3}
{
  \draw[gray, dashed, thin]
    (\x,3) .. controls +(0.3,0.6) and +(0.3,-0.6) .. (\x,0);
}
\draw[->, thick] (0,0) -- (3.6,0) node[anchor=north east]{$x$};
\draw[->, thick] (0,0) -- (0,3.6) node[anchor=south east]{$y$};

\fill[red] (1,0) circle (2pt);   
\fill[blue] (0,0) circle (2pt);  
\fill[blue] (1,1) circle (2pt);  
\fill[blue] (2,0) circle (2pt);  
\fill[blue] (1,3) circle (2pt);  
\end{tikzpicture}
\caption{Two-dimensional lattice with \(L=4\), showing periodic boundary conditions. The red point highlights a site $i$ and the blue points indicate its nearest neighbors $N(i)$.}
\label{Lattice}
\end{figure}
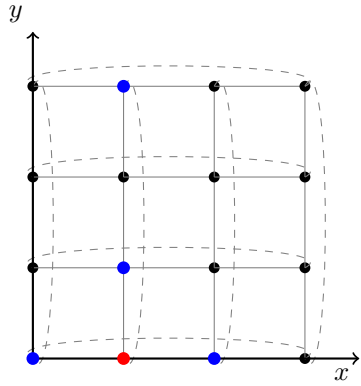

\par\medskip\noindent 
We define the set of lattice sites as
\begin{equation*}
    I := \bigl\{1,\dots,L^d\}.
\end{equation*}
For each site $i\in I$, we denote by $N(i)$ the set of its nearest neighbors. Each site has exactly $2d$ nearest neighbors, i.e. $|N(i)| = 2d$ for all $i \in I$.

\par\medskip\noindent
The objective function of the model in terms of angular variables is given by:
 \begin{equation}
 \label{eqn:function}
    f(\theta) = - \frac{1}{2}\sum_{i\in I}\sum_{j\in N(i)}\cos(\theta_i-\theta_j) - \Delta \sum_{i\in I} \cos(\theta_i-\phi_i),
\end{equation}
  where $\theta=(\theta_1,\dots,\theta_{L^d})^T$ is a vector containing the $L^d$ angles $\theta_i$.  In addition,
 \begin{itemize}
    \item $\phi_i \sim \mathcal{U}(0, 2\pi)$ denotes the orientation of the random field at site $i$ (kept fixed during the energy minimization process);
    \item $\Delta>0$ denotes the so-called disorder strength.
\end{itemize}
\par\smallskip\noindent
Then, the unconstrained optimization problem is the following:
\begin{equation}\label{unconstrained}
    \min_{\theta \in \mathbb{R}^{L^d}} f(\theta).
\end{equation}

 \par\medskip\noindent
We now discuss some preliminary observations on the objective function defined in \eqref{eqn:function}, analyzing certain cases in which a good approximation of the global minimum can be obtained. Note that there exists at least one solution, as $f$ is a continuous function and can be restricted, without any loss of generality, to the compact set $[0,2\pi]^{L^d}$.

\par\medskip\noindent
First, the objective function $f$ can be expressed as a weighted sum of two terms:  $f(\theta)=f_1(\theta)+\Delta f_2(\theta;\phi)$, where
\begin{equation*}
    f_1(\theta)=-\frac{1}{2}\sum_{i \in I}\sum_{j\in N(i)}\cos(\theta_i-\theta_j), \quad\quad f_2(\theta;\phi)=-\sum_{i \in I} \cos(\theta_i-\phi_i).
\end{equation*} 
The two terms $f_1$ and $f_2$ are respectively bounded by:
\begin{equation}
\label{eqn:bounds}
    -dL^d  \le f_1(\theta) \le dL^d, \quad\quad -L^d \le f_2(\theta;\phi) \le L^d.
\end{equation} 
Let $f_{\min}$ be the global minimum of $f$. By \eqref{eqn:bounds}, $f_{\min}$ is lower-bounded by the following quantity:
\begin{equation}
\label{eq:flow}
f_{\min}\ge f_{low}:=-(d+\Delta)L^d.
\end{equation}
In addition, the lower-bounds in \eqref{eqn:bounds} are attained by the respective minimizers of the two functions. Moreover,
\begin{itemize}
\item for all $i\in I$ and $j \in N(i)$, $\cos(\theta_i-\theta_j)$ is maximized when $\theta_i=\theta_j$ (mod $2\pi$); hence,  the minimum of the first function is attained when all variables coincide (mod
$2\pi$):
\begin{equation*}
\argmin_{\theta \in \mathbb{R}^{L^d}} f_1(\theta) = \{\theta=(\theta_1,\dots,\theta_{L^d})^T : \exists\bar{\theta}\in\mathbb R \text{ s.t. } \theta_i=\bar{\theta}\mod 2\pi\quad\forall i\in I\}; 
\end{equation*}
\item for all $i\in I$, $\cos(\theta_i-\phi_i)$ is maximized when $\theta_i=\phi_i$ (mod $2\pi$); hence,  the minimizer of the second function is equal to the external field $\phi$ (mod
$2\pi$):
\begin{equation*}
\argmin_{\theta \in \mathbb{R}^{L^d}} f_2(\theta;\phi) = \{\theta=(\theta_1,\dots,\theta_{L^d})^T : \theta_i=\phi_i\mod 2\pi \quad\forall i\in I\}.
\end{equation*}
\end{itemize}

\begin{remark}
\label{rem:invariance}
There is a unique (mod $2\pi$) minimizer of $f_2$, but infinite minimizers of $f_1$. Indeed, $f_1$ is rotational invariant: consider two vectors $\theta^a,\theta^b\in \mathbb R^{L^d}$ such that $\theta^a_i=\theta_i^b+\bar{\theta} \mod 2\pi$ for all $i\in I$ (for some $\bar\theta\in\mathbb R$), then:
\begin{equation*}
    f_1(\theta^a)=-\frac{1}{2}\sum_{i \in I}\sum_{j\in N(i)}\cos(\theta_i^a-\theta_j^a)=-\frac{1}{2}\sum_{i \in I}\sum_{j\in N(i)}\cos(\theta_i^b+\cancel{\bar{\theta}}-\theta_j^b-\cancel{\bar{\theta}})=f_1(\theta^b).
\end{equation*} 
Conversely, function $f_2$ is not rotational invariant. For instance, take $\bar{\theta}=\pi$; then,
\begin{equation*}
    f_2(\theta^a;\phi)=-\sum_{i\in I}\cos(\theta_i^a-\phi_i)=-\sum_{i\in I}\cos(\pi+\theta_i^b-\phi_i)=\sum_{i\in I}\cos(\theta_i^b-\phi_i)=-f_2(\theta^b;\phi).
\end{equation*} 
As a consequence, the objective function $f$ is not rotational invariant, unless we take $\Delta=0$.
\end{remark}
\par\smallskip\noindent
In general, $f_{\min}$ cannot be found by directly minimizing $f_1$ and $f_2$, as they cannot be minimized simultaneously. However, for sufficiently small or  sufficiently large values of $\Delta>0$,  only one term is prevailing. In these cases, we aim to show that the solutions found by minimizing that term are good approximations of the global minimum $f_{\min}$ in terms of optimality gap, as in the following definition:

\begin{definition}
\label{def:eps-glob-min}
Let $\theta \in \mathbb R^{L^d}$, and $\epsilon>0$. Let $f_{\min}$ be the global minimum of $f$. We say that $\theta$ is an \emph{$\epsilon$-global minimizer} of $f$ if
\begin{equation*}
\frac{f(\theta)-f_{\min}}{|f_{\min}|} < \epsilon.
\end{equation*}
\end{definition}
\noindent
Let $\epsilon \in(0,1)$, define the following two quantities:
\begin{equation*}
    \Delta_{\epsilon}^{(1)}:= d\frac{\epsilon}{2-\epsilon}, \quad\quad
    \Delta_{\epsilon}^{(2)}:= d\frac{2-\epsilon}{\epsilon}.
\end{equation*}
\begin{proposition}
Let $\epsilon \in(0,1)$, $\theta^{(1)} \in \argmin f_1(\theta)$, and $\theta^{(2)} \in \argmin f_2(\theta;\phi)$.
\begin{enumerate}
    \item If $0<\Delta<\Delta_{\epsilon}^{(1)}$, then $\theta^{(1)}$ is an $\epsilon$-global minimizer of $f$;
    \item If $\Delta>\Delta_{\epsilon}^{(2)}$, then $\theta^{(2)}$ is an $\epsilon$-global minimizer of $f$.
\end{enumerate}
\end{proposition}
\begin{proof}
Take $\theta^{(1)} \in \argmin f_1(\theta)$, and $0<\Delta<\Delta_{\epsilon}^{(1)}$. Then, $f_1(\theta^{(1)})=\min f_1=-dL^d$. Moreover, by \eqref{eqn:bounds},  $f_2(\theta^{(1)};\phi)\le L^d$. Then, $f(\theta^{(1)})\le(-d+\Delta) L^d$. Also using \eqref{eq:flow}, it follows:
\[\frac{f(\theta^{(1)})-f_{\min}}{|f_{\min}|}\le\frac{f(\theta^{(1)})-f_{low}}{|f_{low}|}\le\frac{(-\cancel{d}+\Delta+\cancel{d}+\Delta)\cancel{L^d}}{(d+\Delta)\cancel{L^d}}=\frac{2\Delta}{d+\Delta}<\epsilon,\]
where the last inequality follows from $\Delta< d\frac{\epsilon}{2-\epsilon}$. Now, take $\theta^{(2)} \in \argmin f_2(\theta;\phi)$, and $\Delta>\Delta_{\epsilon}^{(2)}$. Then, $f_2(\theta^{(2)};\phi)=\min f_2=-L^d$. Moreover, by \eqref{eqn:bounds},  $f_1(\theta^{(2)})\le dL^d$. Then, $f(\theta^{(2)})\le(d-\Delta) L^d$. Also using \eqref{eq:flow}, it follows:
\[\frac{f(\theta^{(2)})-f_{\min}}{|f_{\min}|}\le\frac{f(\theta^{(2)})-f_{low}}{|f_{low}|}\le\frac{(d-\cancel{\Delta}+d+\cancel{\Delta})\cancel{L^d}}{(d+\Delta)\cancel{L^d}}=\frac{2d}{d+\Delta}<\epsilon,\]
where the last inequality follows from $\Delta> d\frac{2-\epsilon}{\epsilon}$.
\qedhere
\end{proof}

\par\medskip\noindent
The previous result does not provide an optimal solution. However, for $\Delta\in (0,\Delta_{\epsilon}^{(1)})\cup(\Delta_{\epsilon}^{(2)},\infty)$ it provides an $\epsilon$-global minimizer according to Definition \ref{def:eps-glob-min}, thereby making the problem tractable, at least up to numerical precision.
Instead, if
$\Delta\in [\Delta_{\epsilon}^{(1)},\Delta_{\epsilon}^{(2)}]$, the problem is not straightforward, as the two terms play comparable roles. In particular, define 
\begin{equation*}
\bar{\Delta}:=\frac{\min f_1}{\min f_2}=\frac{-dL^d}{-L^d}=d;
\end{equation*}
it is reasonable that for values of $\Delta$ close to $\bar{\Delta}$, the contributions of the two terms may be very similar. For instance, for $\Delta=\bar \Delta$, $f_1(\theta^{(1)})=\Delta f_2(\theta^{(2)};\phi)=-\bar \Delta L^d$. Proliferation of local minima is expected, with this phenomenon presumably being more noticeable for values of $\Delta\in[\Delta_{\epsilon}^{(1)},\bar \Delta]$ than for  $\Delta\in(\bar\Delta,\Delta_{\epsilon}^{(2)}]$, since, as already noted, there are infinite minimizers of $f_1$ but a unique minimizer of $f_2$.

\section{The equivalent optimization problem on Riemannian manifold}
\label{sec:manifolds}
For each site $i\in I$ we consider two-component unit vectors $x_i=(\cos\theta_i,\sin\theta_i)^T$ and $h_i=(\cos\phi_i,\sin\phi_i)^T$.
We introduce the following $2 \times L^d$ matrices:
\[X:=\bigg[x_1\bigg|\dots\bigg|x_{L^d}\bigg] \quad \text{and} \quad
H:=\bigg[h_1\bigg|\dots\bigg|h_{L^d}\bigg] \in \mathbb R^{2 \times L^d},\]
where every column of $X$ and $H$ has unit norm, i.e., $\|h_i\|=\|x_i\|=1$.
We can rewrite the objective function in \eqref{eqn:function} as
\begin{equation}
	\label{objf}
    f(X):=-\frac{1}{2}\sum_{i\in I}\sum_{j\in N(i)}x_j^Tx_i-\Delta \sum_{i\in I}h_i^Tx_i.
\end{equation}
Then, we consider the following optimization problem
\begin{equation}\label{manifold}
\min_{X\in \mathcal M}f(X),
\end{equation}
where $\mathcal M$ is the \textit{oblique manifold} defined as follows:
$$
\mathcal M
=\bigg\{X=\bigg[x_1\bigg|\dots\bigg|x_{L^d}\bigg]\in\mathbb R^{2\times L^d} : \|x_i\|=1 \text{ for all } i=1,\dots,L^d\bigg\}.
$$
Note that $\mathcal M$
is the Cartesian product of $L^d$ spheres $S^{1}$, i.e., $\mathcal M=S^{1}\times\ldots\times S^{1}$.

\par\medskip\noindent
The Euclidean gradient takes the following form  \[\nabla f(X)=\bigg[\nabla_1 f(X)\bigg|\dots\bigg|\nabla_{L^d} f(X)\bigg]\in\mathbb{R}^{2\times L^d},\] where:
\begin{equation*}
\nabla_i f(X) = -\sum_{j \in N(i)}x_j - \Delta h_i.
\end{equation*}
We introduce the Euclidean Hessian operator $\mathcal H_f(X): \mathbb R^{2 \times L^d} \to \mathbb R^{2 \times L^d}$. For every $V=\big[v_1\big|\dots\big|v_{L^d}\big] \in \mathbb R^{2 \times L^d}$, we have
\[\mathcal H_f(X)[V]=\bigg[\nabla^2_1 f(X)[V]\bigg|\dots\bigg|\nabla^2_{L^d} f(X)[V]\bigg]\in\mathbb{R}^{2\times L^d},\] where:
\begin{equation*}
\nabla^2_i f(X)[V] = -\sum_{j \in N(i)}v_j.
\end{equation*}
As the operator does not depend on $X$, we can lighten the notation and only write $\mathcal H_f$.

\begin{remark}\label{indef}
The Hessian operator $\mathcal H_f: \mathbb R^{2 \times L^d} \to \mathbb R^{2 \times L^d}$ is indefinite.
Indeed, for every $V\in \mathbb R^{2\times L^d}$ we can write
\begin{equation*}
\langle V, \mathcal H_f [V]\rangle_{F}=-\sum_{i\in I}\sum_{j\in N(i)}v_i^Tv_j.
\end{equation*}
Consider a matrix $V=\big[v_1\big|\dots\big|v_{L^d}\big]\in \mathbb{R}^{2 \times L^d}$, where $v_i=\bar{v}\in \mathbb{R}^2$ ($\bar{v}\not=0$) for all $i=1,\dots,L^d$.
Then, we obtain $\langle V, \mathcal H_f[V]\rangle_{F}=-2dL^d\bar v^T\bar v<0$.
Now, choose two indices $\bar i$ and $\bar j$ such that $\bar j \in N(\bar i)$ and take a matrix $V=\big[v_1\big|\dots\big|v_{L^d}\big]\in \mathbb{R}^{2 \times L^d}$, where $v_{\bar i}=\bar{v}$, $v_{\bar j}=-\bar{v}$ (for some $\bar{v}\not=0$) and $v_i=0$ for all $i\in\{1,\dots,L^d\}\setminus\{\bar i, \bar j\}$. We obtain $\langle V, \mathcal H_f[V]\rangle_{F}=-2v_{\bar i}^Tv_{\bar j}=2\bar v^T\bar v>0$. Therefore, we can conclude that $\mathcal H_f$ is indefinite. As a consequence, we have that the function $f$ defined by (\ref{objf}) is neither concave nor convex over $\mathbb R^{2 \times L^d}$.
\end{remark}

\par\noindent
For the theory and algorithms of optimization on manifolds, we refer the reader to the following references
 \cite{boumal2023intromanifolds,cartis,MatrixManifolds}.

\section{Problem-specific Basin Hopping}
\label{sec:BH}

In this section, we first present the Basin Hopping (BH) \cite{Wales1998, Locatelli2013} framework and then we propose three perturbation strategies specifically designed for the optimization problem under consideration. Each strategy defines a different BH variant, which is subsequently extended to a corresponding Population Basin Hopping (PBH) algorithm, in which a set of candidate solutions is jointly evolved instead of a single one in order to assess whether the introduction of population-based mechanisms provides an improvement over the corresponding single-point variants.\\

\subsection{Algorithmic scheme}
In Algorithm \ref{alg:MBH}, we report the scheme of the BH method which constitutes the global optimization framework adopted in this work. After a local solution is obtained, BH generates a perturbation of the current point and uses it as a new starting point for the local solver. If the perturbed and locally optimized point yields an improvement, the current point is updated accordingly; otherwise, it is discarded and the procedure is repeated until the stopping criterion is satisfied. In the following, $\mathcal G()$ denotes the generation of a random starting point, $\mathcal L(X,f)$ denotes a local solver applied to $f$ starting from $X$, and $\mathcal P(X)$ denotes a perturbed point generated from $X$.

\begin{algorithm}[H]
\caption{BH Algorithm}
\label{alg:MBH}
\begin{algorithmic}[1]
\State $X \gets \mathcal G()$ \Comment{random starting point}
\State $X \gets \mathcal L(X, f)$ \Comment{local optimum}
\State $f_{\text{best}} \gets f(X)$
\While{\text{stopping criterion not satisfied}} \Comment{start of BH} 
    \State $Y \gets \mathcal P(X)$ \Comment{perturbation}
    \State $Y \gets \mathcal L(Y,f)$ \Comment{local optimization}
    \If{$f(Y) < f(X)$} \Comment{improvement}
        \State $X \gets Y$
        \State $f_{\text{best}} \gets f(X)$ 
    \EndIf
\EndWhile \Comment{end of BH}

\end{algorithmic}
\end{algorithm}

\subsection{Perturbation strategies}

\par\medskip\noindent
Different implementations of the perturbation operator $\mathcal P(x)$ are considered in this work. The first is a standard random perturbation, whereas the other two variants are novel and specifically designed for the problem at hand, exploiting both the structure of the objective function and the underlying geometry of the problem.

\par\medskip\noindent
\textit{Strategy 1: random perturbation.}
The first strategy consists of a simple random perturbation in which the new point is generated by sampling in the neighborhood of the current solution. Let $\eta^{\text{max}}>0$ be a constant, and let $\theta = (\theta_1,\dots,\theta_{L^d})^T \in \mathbb R^{L^d}$ denote the current point in the angular formulation. For each $i=1,\dots,L^d$, let $\eta_i \sim \mathcal N(0, \eta^{\text{max}})$ be an independent Gaussian random variable. The perturbed point is then obtained component-wise as
\begin{equation*}
 \theta^{\text{new}}_i = \theta_i + \eta_i.   
\end{equation*}
\par\medskip\noindent
\textit{Strategy 2: spatial perturbation.} The second strategy introduces spatial correlation into
the perturbation, as opposed
to the site-independent noise of Strategy~1. Let $A \in \mathbb{R}^{L^d \times L^d}$ be the adjacency
matrix of the lattice, i.e.\ $A_{i,j}=1$ if $j\in N(i)$ and $0$ otherwise. We define the row-stochastic \textit{diffusion operator}
\begin{equation*}
W = \frac{1}{2d}A;
\end{equation*}
we trivially observe that, for every $i=1,\dots,L^d$, $\sum_{j=1}^{L^d} W_{i,j}=1$. Moreover, for any vector $v\in\mathbb R^{L^d}$, the $i$-th entry of $W\!v$ is the average of $v_j$ over
$j\in N(i)$: applying $W$ therefore mixes each site's value with that of its neighbors. Starting from an isotropic Gaussian noise vector $\xi \sim \mathcal N(0, \mathcal{I}_{L^d})$, with $\mathcal{I}_{L^d}$ being the identity matrix of size $L^d \times L^d$ and $\xi \in \mathbb
R^{L^d}$, we apply one step of diffusion on the lattice graph to obtain a spatially correlated
fluctuation:
\begin{equation*}
\tilde\eta = W\xi.
\end{equation*}
Since diffusion averages the noise over neighboring sites, its amplitude shrinks with respect to that
of $\xi$; to make the perturbation strength comparable to that of Strategy~1 and independent of the
system size $L^d$, we rescale $\tilde\eta$ with respect to its maximum absolute component, i.e.\ using
the infinity norm:
\begin{equation*}
\eta = \eta^{\text{max}}\cdot\frac{\tilde\eta}{\|\tilde\eta\|_\infty}, \qquad \|\tilde\eta\|_\infty =
\max_{i=1,\dots,L^d}|\tilde\eta_i|.
\end{equation*}
The perturbed point is then obtained on the manifold $\mathcal M$ (rather than in angular form) via a
local rotation of each spin: recalling that every $x_i=(\cos\theta_i,\sin\theta_i)^T$ satisfies
$\|x_i\|=1$, the perturbation acts as the rotation
\begin{equation*}
x_i^{\text{new}} = R(\eta_i)\,x_i = \begin{pmatrix}\cos\eta_i & -\sin\eta_i\\ \sin\eta_i &
\cos\eta_i\end{pmatrix}x_i.
\end{equation*}
Since $R(\eta_i)$ is a rotation matrix, $\|x_i^{\text{new}}\|=\|x_i\|=1$ for every $i=1,\dots,L^d$, so
the perturbed point lies on $\mathcal M$ by construction and no retraction step is required after the
perturbation.
\par\medskip\noindent
Note that Strategies 1 and 2 share the same underlying identity,
$R(\eta_i)x_i = (\cos(\theta_i+\eta_i), \sin(\theta_i+\eta_i))^T$, so casting the perturbation as a rotation on $\mathcal M$ rather
than as an angular shift is not itself a source of difference. The real differences are the diffusion
step, absent in Strategy~1 (equivalent to $\tilde\eta=\xi$), and the resulting need for the
infinity-norm normalization: since $W$ averages neighboring entries, the extremes of $\tilde\eta$
shrink as $L^d$ grows, so rescaling by $\|\tilde\eta\|_\infty$ is required to keep $\eta^{\text{max}}$ a
size-independent perturbation scale, as it already is in Strategy~1.

\par\medskip\noindent
\textit{Strategy 3: alternating perturbation.}
The third strategy exploits the physical structure of the problem by accounting for both the interaction among neighboring lattice sites and the site-dependent external field. Rather than perturbing each angle by an independent random increment, this strategy displaces each angle towards a target direction, alternating between two rules at successive calls of $\mathcal P(x)$.
Let $\eta^{\text{max}}>0$ be fixed and let $\eta_i\sim\mathcal U(0,\eta^{\mathrm{max}})$ independently for $i=1,\ldots,L^d$. At even applications of $\mathcal P(x)$, each angle is displaced toward the circular mean of the neighboring spins according to
\begin{equation*}
\theta^{\text{new}}_i = \theta_i + \eta_i(\bar\theta_i - \theta_i), \qquad i = 1,\dots,L^d,   
\end{equation*}
where
\begin{equation*}
    \bar\theta_i = \operatorname{atan2}\!\left(\frac{1}{|{N}(i)|}\sum_{j\in{N}(i)}\sin\theta_j,\; \frac{1}{|{N}(i)|}\sum_{j\in{N}(i)}\cos\theta_j\right). 
\end{equation*}
At odd applications, each angle is instead displaced toward the corresponding external field direction,
\begin{equation*}
    \theta^{\text{new}}_i = \theta_i + \eta_i(\phi_i - \theta_i), \qquad i = 1,\dots,L^d.
\end{equation*}
This alternation is designed to balance the exploration induced by the local coupling structure of the lattice with that induced by the external field. Consequently, the generated perturbations remain consistent with the underlying physical model while preserving sufficient diversity to facilitate escapes from poor local minima.

\subsection{Population-based variants}

The perturbation strategies proposed for the BH framework were also evaluated within the population-based version of the algorithm. In Algorithm \ref{alg:PBH}, we report the scheme of the PBH method \cite{Grosso2007, Locatelli2013}, an extension of BH that operates on a population of candidate solutions rather than on a single point. The method starts from a population $\mathcal{X}$ of randomly generated starting points, each of which is locally optimized to obtain an initial population of local optima. At each iteration, every point $X_p \in \mathcal{X}$ is perturbed and locally optimized to produce a candidate point $Y_p$, which is collected into a candidate set $\mathcal{Y}$; this step plays the same role as the perturbation-and-local-optimization step of BH, but is applied ``in parallel'' to every member of the population rather than to a single current point. 

\begin{algorithm}[H]
\caption{PBH Algorithm}
\label{alg:PBH}
\begin{algorithmic}[1]
\State $\mathcal{X} \gets \{X_p \mid X_p = \mathcal{G}()\}$ \Comment{random starting population}
\ForAll{$X_p \in \mathcal{X}$}
    \State $X_p \gets \mathcal L(X_p, f)$ \Comment{local optimum}
\EndFor
\State $f_{\text{best}} \gets \min\{f(X_p) \mid X_p \in \mathcal{X}\}$
\While{\text{stopping criterion not satisfied}} \Comment{start of PBH} 
    \State $\mathcal{Y} \gets \emptyset$ \Comment{initialization of the candidate points sets}
    \ForAll{$X_p \in \mathcal{X}$}
        \State $Y_p \gets \mathcal P(X_p)$ \Comment{perturbation}
        \State $\mathcal{Y} \gets \mathcal{Y} \cup \{\mathcal L(Y_p,f)\}$ \Comment{local optimization}
    \EndFor
    \State $\mathcal{X} \gets \textit{Select}(\mathcal{X}, \mathcal{Y})$ \Comment{selection of the points for the next population}
    \State $f_{\text{best}} \gets \min\{f(X_p) \mid X_p \in \mathcal{X}\}$
\EndWhile \Comment{end of PBH}

\end{algorithmic}
\end{algorithm}
\par\medskip\noindent
Once the candidate set $\mathcal{Y}$ has been generated, the population for the next iteration is obtained through a selection step, whose scheme is reported in Algorithm \ref{alg:PBH_Select}. In this step, we determine, for each candidate $Y_p \in \mathcal{Y}$, the nearest point $X_c$ in the current population with respect to the geodesic distance $d(\cdot,\cdot)$ on $\mathcal M$ \cite{boumal2023intromanifolds}, i.e.,
\begin{equation}
\label{eq::geodesic_distance}
d(X,Y) = \sqrt{\sum_{i=1}^{L^d} \arccos\!\big(x_i^T y_i\big)^2};
\end{equation}
if this dissimilarity exceeds a threshold $D_{\mathrm{cut}}$, $X_c$ is instead replaced by the worst point of the population in terms of objective function. In this way, candidates close to existing solutions compete locally against their nearest neighbor, while genuinely novel candidates compete against the weakest member of the population. In either case, $X_c$ is replaced by $Y_p$ only if the latter yields an improvement. This mechanism allows PBH to maintain diversity within the population, favoring the exploration of multiple regions of the energy landscape simultaneously, while still allowing improving candidates to displace weaker solutions regardless of their similarity to the current population. The entire PBH procedure is repeated until a stopping criterion is satisfied.

\begin{algorithm}[H]
\caption{\textit{Select} function for PBH}
\label{alg:PBH_Select}
\begin{algorithmic}[1]
\ForAll{$Y_p \in \mathcal{Y}$}
    \State $c \in \argmin_{i \in \{1,\ldots,|\mathcal{X}|\}} d(X_i, Y_p)$ \Comment{$d(\cdot, \cdot)$: geodesic distance on $\mathcal{M}$} 
    \If{$d(X_c, Y_p) > D_{\text{cut}}$} \Comment{$D_{\text{cut}}$: threshold distance}
    \State $c \in \argmax_{i \in \{1,\ldots,|\mathcal{X}|\}} f(X_i)$
    \EndIf
    \If{$f(Y_p) < f(X_c)$} \Comment{improvement}
    \State $\mathcal{X} \gets \mathcal{X} \setminus \{X_c\} \cup \{Y_p\}$
    \EndIf
\EndFor

\end{algorithmic}
\end{algorithm}

\section{Computational study}

\label{sec:results}

In this section, we report the results of  computational experiments aimed at assessing the quality and consistency of the proposed global optimization approaches for the Riemannian reformulation of the considered problem. The implementation of the presented algorithms can be found in the GitHub repository\footnote{available at \href{https://github.com/LorenzoCiarpa/RFXYBH-Riemann}{\tt https://github.com/LorenzoCiarpa/RFXYBH-Riemann}}, where all experimental outputs and the considered external fields can be downloaded, along with a table with all putative global minima found for each configuration. All experiments were performed on a machine running Ubuntu 24.04 OS, equipped with an Intel(R) Core(TM) i5-10600KF processor (6 cores, 4.10 GHz) and 32 GB of RAM.
\par\medskip\noindent
The experimental setup considers three-dimensional lattices ($d=3$) with system sizes $L\in\{10,15,20,25,32\}$, different values of the parameter $\Delta \in \{1.0,2.0,2.5,3.0,4.0\}$, and five different realizations of the external field, denoted by $h_i$, $i=1,\dots,5$. Each realization is generated by independently sampling the field values from the uniform distribution $\mathcal{U}(0,2\pi)$.
\par\medskip\noindent
The BH and PBH frameworks have been tested using all three perturbation strategies introduced in Section \ref{sec:BH}. In the experimental results, these variants are identified through a suffix appended to the algorithm name: ``rnd'' for the random-based perturbation strategy (Strategy 1), ``spa'' for the spatial perturbation strategy (Strategy 2), and ``alt'' for the alternating perturbation strategy (Strategy 3). While PBH naturally maintains a population of solutions, BH has been evaluated in a multistart configuration. The proposed approaches have been compared with a MultiStart (MS) version of the Riemannian local solver, the Differential Evolution (DE) \cite{Storn1997} evolutionary algorithm, which is widely recognized in the literature as one of the most effective evolutionary approaches for highly irregular optimization problems, and two heuristics commonly adopted in the physics literature to solve the standard unconstrained formulation of the problem, namely Simulated Annealing (SA) \cite{doi:10.1126/science.220.4598.671} and Parallel Tempering (PT) \cite{Hukushima1996,marinari1992}, both implemented in a multistart configuration. 
\par\medskip\noindent
The algorithmic parameters were selected through preliminary experiments, not reported here for the sake of brevity and performed on a subset of the considered problem instances. For the three (P)BH perturbation strategies, the maximum perturbation magnitude was set to $\eta_{\text{max}}=0.5$ for Strategy 1, $\eta_{\text{max}}=2$ for Strategy 2, and $\eta_{\text{max}}=0.3$ for Strategy 3. The different values of $\eta_{\text{max}}$ are expected, since the three perturbation mechanisms differ substantially. For Strategy 3, a more conservative value was found to be preferable because the perturbations explicitly exploit structural information from both components of the objective function, making them inherently more directed. In contrast, Strategy 2 benefits from a larger value of $\eta_{\text{max}}$, as the perturbation is subsequently normalized after the diffusion step, which effectively controls its magnitude while preserving the induced spatial correlations. Finally, for the purely random perturbation of Strategy 1, an intermediate value of $\eta_{\text{max}}=0.5$ provided the best balance between exploration and preserving the quality of the current solution.
For the PBH variants, the cutoff distance $D_{\text{cut}}$ was dynamically defined according to the initial population, namely as the average value of the geodesic distance \eqref{eq::geodesic_distance} computed over all pairs of solutions belonging to the initial population. For all BH and PBH variants and MS, we employed the \texttt{trust\_regions} routine provided by the Python package \texttt{pymanopt} as the local Riemannian optimizer \cite{pymanopt}, as discussed in Appendix \ref{sec:app_local}; a local search was terminated when the norm of the Riemannian gradient became lower than or equal to $10^{-6}$ or when the number of inner iterations reached $10,\!000$. For DE, the mutation and crossover parameters were set to $F=0.1$ and $CR=0.5$, respectively, while the threshold parameter used to determine whether the population had collapsed to a single solution was set to $10^{-4}$. For SA, the initial temperature was set to $T=T_{\text{max}}=2$, and it was decreased at each iteration according to $T_{it+1}=0.99T_{it}$; once the minimum temperature $T_{\text{min}}=10^{-3}$ was reached, the algorithm was allowed to terminate. For PT, the minimum and maximum temperatures were set to $T_{\text{min}}=0.1$ and $T_{\text{max}}=2.5$, respectively. Except for the multistart solver, all algorithms were executed using an initial population of $10$ solutions, whose values were randomly sampled from the uniform distribution $\mathcal{U}(0,2\pi)$. The MS approach was instead initialized with $500$ additional random solutions. 
\par\medskip\noindent
Since all tested algorithms include random components, each method was executed $5$ times on every problem instance using five different seeds for the pseudo-random number generator. Due to the structural differences among the tested approaches, each run was limited to a maximum execution time of $2$ minutes, unless an algorithm-specific stopping criterion was reached earlier indicating that no further improvement of the current solutions was expected. Since DE, SA, and PT do not employ any local solver, these methods were granted an additional two-minute execution budget to compensate for the absence of local refinement. 
\par\medskip\noindent
The algorithms were mainly compared according to two metrics: the objective value of the returned solution, denoted as $f^\star$, and the number of iterations performed by the local solver, normalized with respect to the number of iterations executed by each algorithm and denoted as $it_{ls}^{m}$.
\par\medskip\noindent
To summarize the results in terms of solution quality, we report the cumulative distribution of relative gap to the optimal value $\frac{|f-f^\star|}{|f^\star|}$. More specifically, we consider the cumulative distribution function of the relative gap between the average score achieved by each solver over the five independent runs and the best score obtained by any solver in any run. This type of representation is more suitable than standard performance profiles \cite{Dolan2002}, since the considered metric may assume both positive and negative values and does not represent an absolute computational cost. Performance profiles are instead employed to provide a compact comparison of the algorithms in terms of $it_{ls}^{m}$. Following the approach in \cite{Lapucci2026}, and in order to highlight the sensitivity of the solvers to stochastic effects, we additionally report cumulative distributions and performance profiles based on the best and worst outcomes obtained by each solver over the five independent runs. These curves define a shaded region around the distribution based on the average performance, providing a visual representation of the variability induced by the random initialization and random components of the algorithms. For all plots, the reference value used to construct the distributions and performance measures corresponds to the overall best solution obtained among all tested algorithms and all independent runs. 
\par\medskip\noindent
Finally, we employ an additional graphical representation showing, for a given number of iterations, the percentage of problem instances for which each algorithm achieves a better solution than its competitors. This analysis provides further insight into the practical efficiency and effectiveness of the tested methods throughout the optimization process, complementing the comparisons based only on the final solutions.

\subsection{Computational efficiency: manifold versus unconstrained formulation}\label{sec:computational efficiency}

In this section, we report results motivating our choice of the Riemannian manifold reformulation of the considered problem with respect to the standard unconstrained formulation. This first analysis is, in our opinion, particularly relevant, as it provides empirical evidence of the advantages offered by the proposed Riemannian formulation. 
\par\medskip\noindent
Since both formulations optimize the same objective function but rely on different geometric representations, this comparison highlights the computational advantages provided by the manifold approach. In the unconstrained case, the evaluation of the cost function, the (Euclidean) gradient, and the (Euclidean) Hessian applied to a vector requires the computation of trigonometric functions, whereas in the manifold approach these quantities involve mainly dot products between vectors. We further observe that the manifold formulation is quadratic, so that its Euclidean Hessian is constant and independent of the current point. 
Moreover, its feasible set is compact, so the iterates remain bounded by construction. By contrast, the angular formulation is periodic and unbounded in the variables. In the latter, nothing prevents the iterates from drifting towards large angular values, where the evaluation of trigonometric functions requires an argument reduction step, thus making each evaluation of the objective function and of its derivatives progressively more expensive. Consequently, the evaluations of the cost function, gradient, and Hessian applied to a vector are more expensive in the unconstrained case.
\par\medskip\noindent
Table~\ref{tab:times_f_grad} reports statistics on the mean computation times of these quantities in both the manifold and the unconstrained formulations for $L \in \{10, 20, 32\}$, where each reported value is averaged over 200 independent evaluations. For instance, a $6\times$, $9\times$, and $12\times$ speed-up can be observed in the evaluation of the cost function, gradient, and Hessian applied to a vector, respectively, when moving from the unconstrained to the manifold formulation with $L = 32$.
Moreover, for the manifold formulation we report the total time required to evaluate the Riemannian gradient (which consists of computing the Euclidean gradient and projecting it onto the tangent space) and the Riemannian Hessian (which includes evaluating the Euclidean gradient, the Euclidean Hessian, and additional operations). In addition, since each new point generated by an algorithm involves both evaluating the cost function and performing a retraction, we report the total time for these operations. As shown in the table, even when including these additional operations, the computation times remain lower than those of the unconstrained counterparts.

\begin{table}[t]
\caption{{Mean times (in seconds) of cost function, (Euclidean) gradient and Hessian evaluation in the manifold and in the unconstrained formulation with $L \in \{10, 20, 32\}$. For the manifold formulation, the total time of cost evaluation and retraction, Riemannian gradient and Hessian applied to a vector are reported.}}\label{tab:times_f_grad}
\begin{tabular*}{\textwidth}{@{\extracolsep\fill}cccccccc}
\toprule
 & \multicolumn{2}{c}{\textbf{$L=10$}} & \multicolumn{2}{c}{\textbf{$L=20$}} & \multicolumn{2}{c}{\textbf{$L=32$}} \\\cmidrule{2-3}\cmidrule{4-5}\cmidrule{6-7}
operation & manifold & unconstr. & manifold & unconstr. & manifold & unconstr. \\[1pt]
\midrule
\\
cost & 3.76e-05 & 1.02e-04 & 1.39e-04 & 8.48e-04 & 5.83e-04 & 3.39e-03 \\
cost + retraction & 4.64e-05 & -- & 1.75e-04 & -- & 8.67e-04 & -- \\[5pt]
\hline
\\
Euclidean gradient & 2.69e-05 & 1.02e-04 & 1.01e-04 & 8.46e-04 & 3.95e-04 & 3.72e-03 \\
Riemannian gradient & 3.26e-05 & -- & 1.30e-04 & -- & 5.12e-04 & -- \\[5pt]
\hline
\\
Euclidean Hessian & 2.13e-05 & 1.44e-04 & 8.33e-05 & 1.20e-03 & 4.73e-04 & 5.50e-03 \\
Riemannian Hessian & 5.92e-05 & -- & 2.35e-04 & -- & 1.19e-03 & -- \\[3pt]
\bottomrule
\end{tabular*}
\end{table}

\subsection{Comparison of perturbation mechanisms in BH and PBH}
We first compare the BH variants introduced in Section \ref{sec:BH}, which differ only in the perturbation strategy employed to escape from local minima. The aim of this analysis is to evaluate the effectiveness of each perturbation mechanism and to identify the most competitive configuration to be used in the following comparisons. The same analysis is then repeated for the corresponding PBH variants, in order to verify whether the relative performance of the perturbation strategies is preserved in the population-based setting. Figures \ref{fig:PP_BH}--\ref{fig:PP_PBH} report the results for BH and PBH, respectively, from three complementary perspectives: (a) the cumulative distribution of the relative gap of the final objective value with respect to the best value found by any of the solvers; (b) the performance profile with respect to the average number of local solver iterations per outer iteration; (c) the percentage of problems for which an algorithm attains a better objective value within a given number of iterations. 
For both BH and PBH, the alternating variant generally reaches solutions with better objective values than the other two variants (which appear to perform almost equivalently), as evidenced by both (a) and (c). However, at each outer iteration, it requires a larger number of local solver iterations, resulting in worse performance according to the corresponding performance profile. Nevertheless, note the short scale of the $\tau$-axis: even in the worst case, the average number of local solver iterations required by the alternating variant is less than three times that of the best-performing variant.
Moreover, this behavior should be interpreted from a different perspective. Indeed, not only may the additional computational effort explain the better objective values achieved by the alternating variant, but, since the total computational time is the same for all variants, it also suggests that the other two variants spend more time performing more ineffective perturbations than the alternating variant. 

\begin{figure}
    \subfloat[Cumulative distribution of the relative gap of the final objective value with respect to the best
value found by any of the solvers.]{\includegraphics[width=0.33\textwidth]{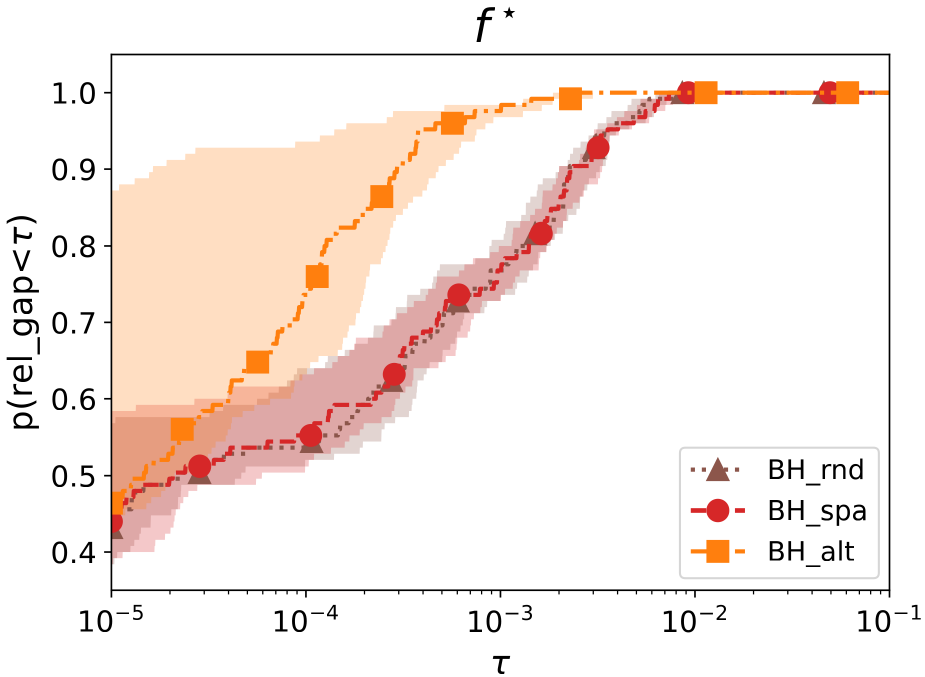}}
    \hfill
    \subfloat[Performance profile with respect to the average number
of local solver iterations per outer iteration.]{\includegraphics[width=0.33\textwidth]{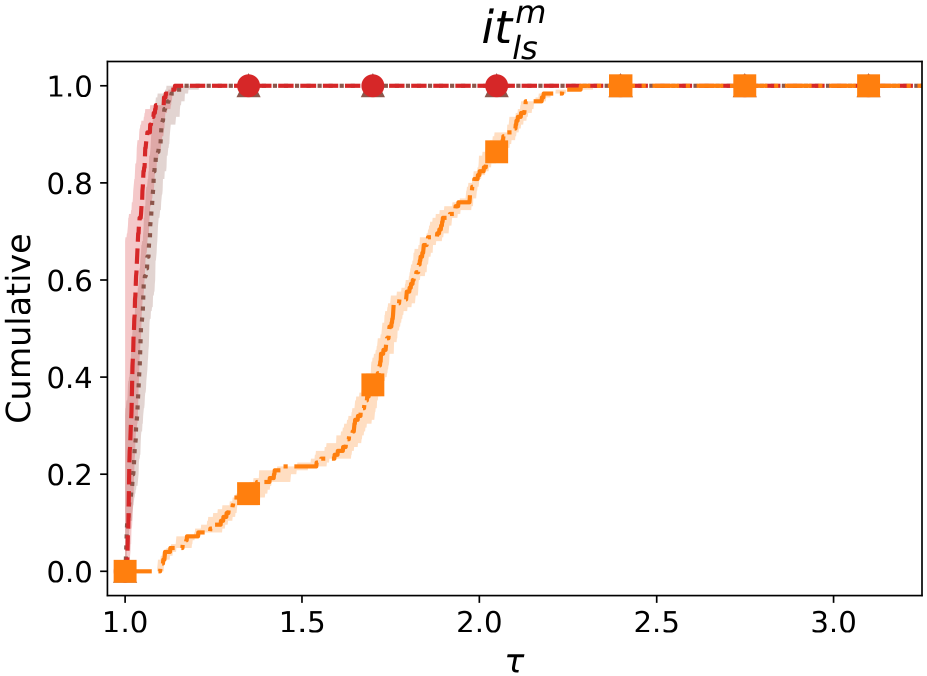}}
    \hfill
    \subfloat[Percentage of problems for which an algorithm attains a better objective value within a given number of iterations.]{\includegraphics[width=0.3025\textwidth]{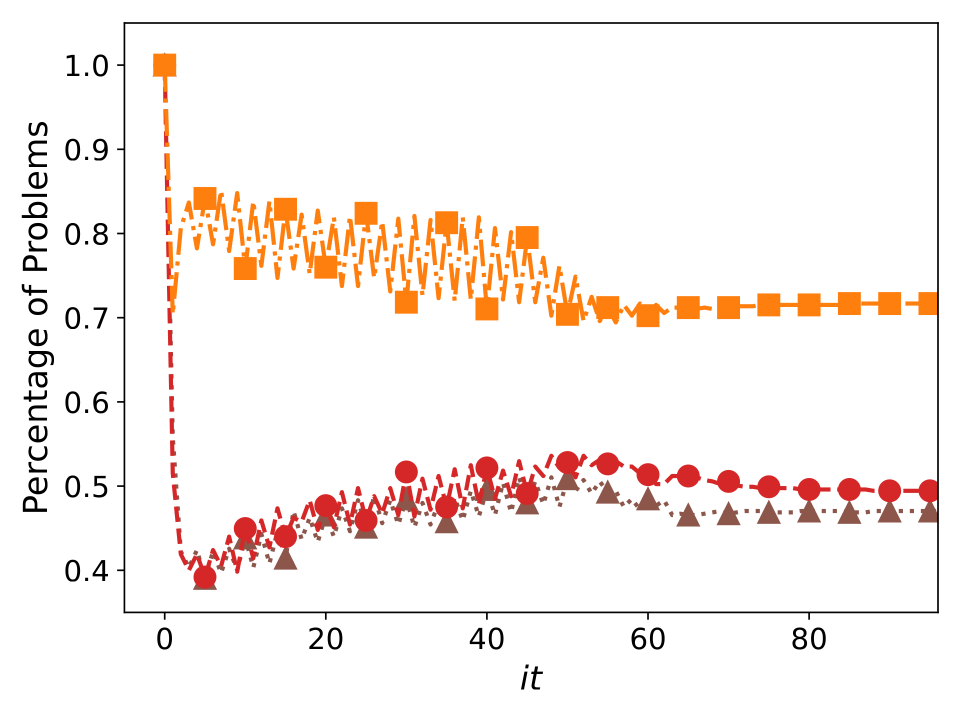}}
    \caption{Results of the comparison between the three variants of Basin Hopping: random (BH\_rnd), spatial (BH\_spa), alternating (BH\_alt).}
    \label{fig:PP_BH}
\end{figure}

\begin{figure}
    \subfloat[Cumulative distribution of the relative gap of the final objective value with respect to the best
value found by any of the solvers.]{\includegraphics[width=0.33\textwidth]{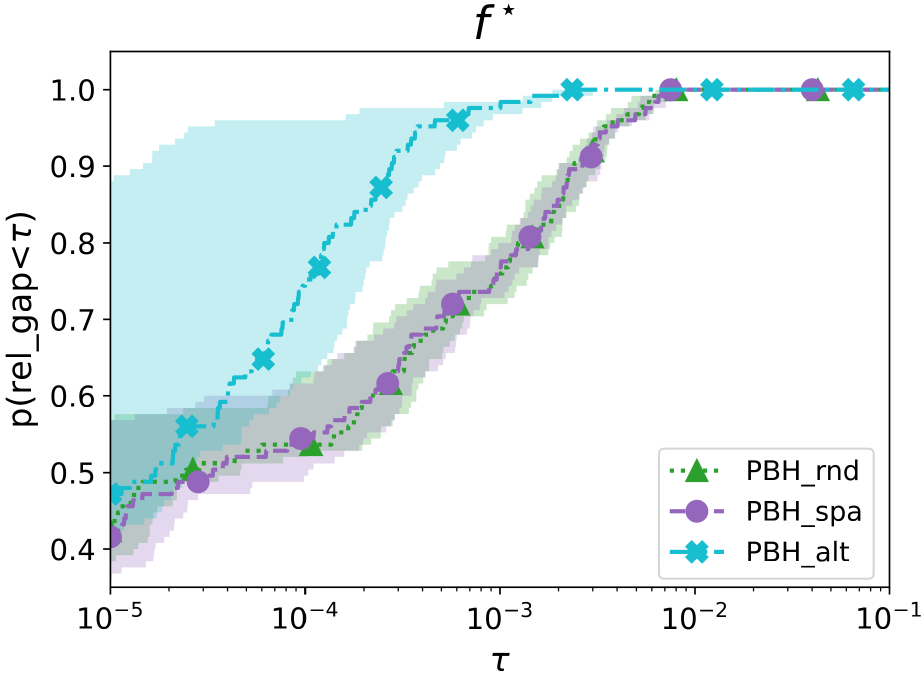}}
    \hfill
    \subfloat[Performance profile with respect to the average number of local solver iterations per outer iteration.]{\includegraphics[width=0.33\textwidth]{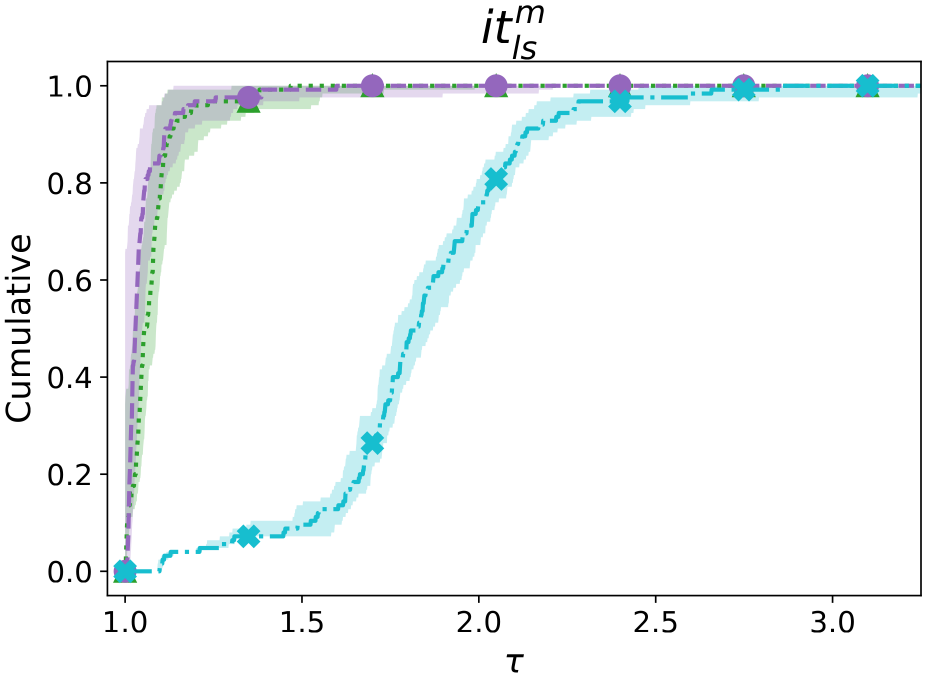}}
    \hfill
    \subfloat[Percentage of problems for which an algorithm attains a better objective value within a given number of iterations.]{\includegraphics[width=0.305\textwidth]{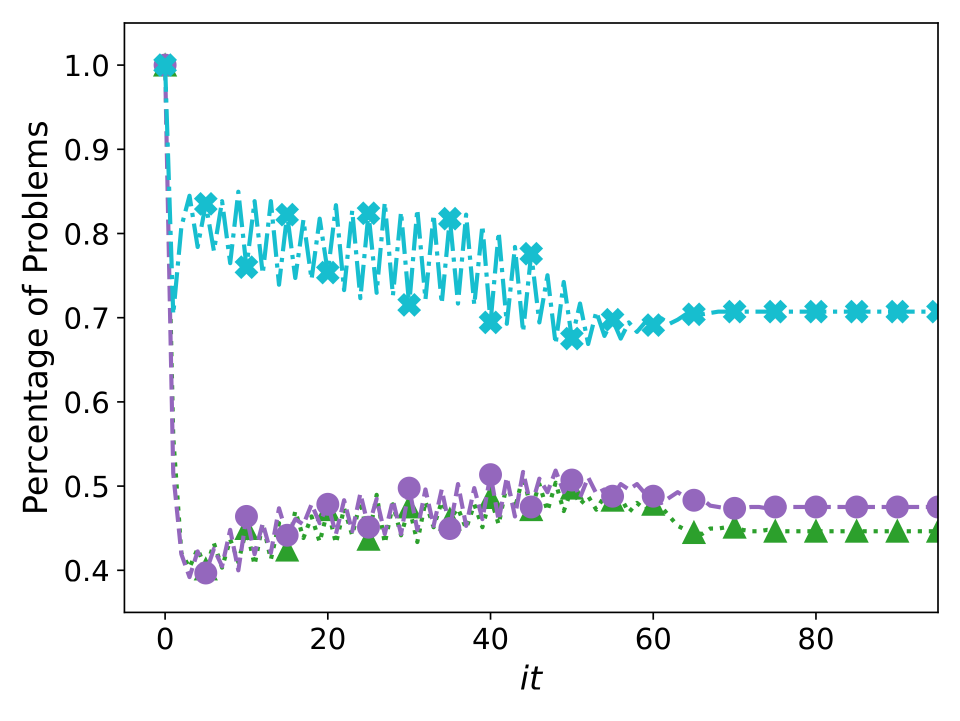}}
    \caption{Results of the comparison between the three variants of Population Basin Hopping: random (PBH\_rnd), spatial (PBH\_spa), alternating (PBH\_alt).}
    \label{fig:PP_PBH}
\end{figure}

\subsection{Comparison with the MultiStart baseline}

We now compare the best-performing BH and PBH variants identified above with a MS strategy. Such a comparison allows us to quantify the benefit provided by the more elaborate exploration mechanisms of BH and PBH with respect to this baseline strategy. We consider the alternating variants (BH\_alt and PBH\_alt, respectively) since, as highlighted in the previous subsection, they generally achieve better objective values than the other variants. Figure \ref{fig:MS} compares MS with BH\_alt and PBH\_alt from two complementary perspectives. On the one hand, the cumulative distribution of the relative gap of the final objective value shows that MS generally attains higher objective values than both BH\_alt and PBH\_alt, whose performances are very similar. On the other hand, the performance profile with respect to the average number of local solver iterations per outer iteration indicates that MS requires substantially more iterations than both BH\_alt and PBH\_alt, with the former performing slightly better than the latter.
The poor performance of MS is also evident from Tables \ref{tab:MS_f3}--\ref{tab:MS_f5}, which report the mean final objective values over the $5$ runs and standard deviations for the external fields $h_3$ and $h_5$, respectively, as representative examples (the results for the other three external fields are analogous and are therefore omitted). It is also worth noting that BH\_alt and PBH\_alt achieve comparable performance throughout the experiments. A slight advantage of BH\_alt emerges only for the largest values of $L$, suggesting that it may be the preferable choice if the approach is to be extended to larger lattice sizes. 

\begin{figure}
    \subfloat[Cumulative distribution of the relative gap of the final objective value with respect to the best value found by any of the solvers.]{\includegraphics[width=0.45\textwidth]{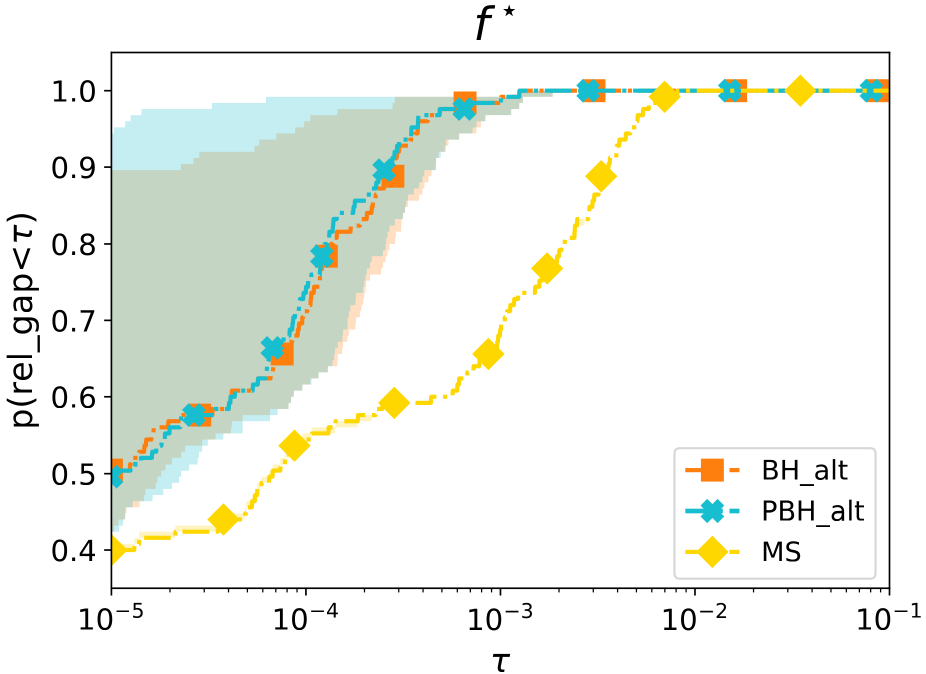}}
    \hfill
    \subfloat[Performance profile with respect to the average number of local solver iterations per outer iteration.]{\includegraphics[width=0.45\textwidth]{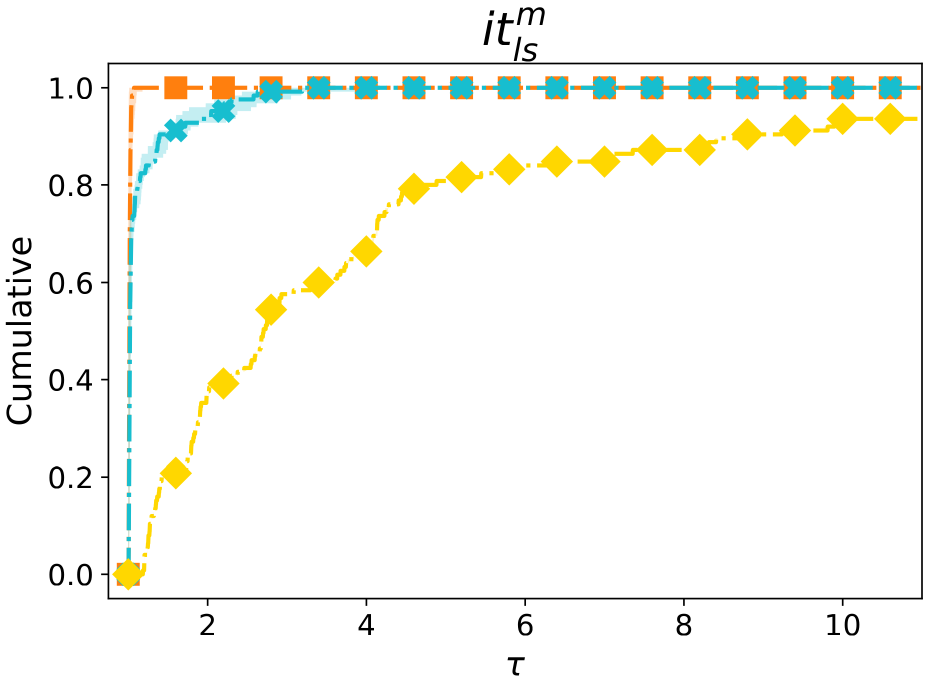}}
    \caption{Results of the comparison between MultiStart (MS) and the alternating variants of Basin Hopping (BH\_alt) and Population Basin Hopping (PBH\_alt).}
    \label{fig:MS}
\end{figure}

\begin{table}[t]
\caption{Comparison on the external field $h_3$ between the mean final function values and standard deviation attained by MultiStart (MS), and the alternating variants of Basin Hopping (BH\_alt) and Population Basin Hopping (PBH\_alt). The lowest mean values are in bold.}
\label{tab:MS_f3}
\begin{tabular*}{\textwidth}{@{\extracolsep\fill}ccccc}
\toprule%
$L$ & $\Delta$ &BH\_alt&PBH\_alt&MS\\%
\midrule
10 & 1.0&\textbf{-3082.89 ($\pm$ 0.00)}&\textbf{-3082.89 ($\pm$ 0.00)}&\textbf{-3082.89 ($\pm$ 0.00)}\\%
& 2.0&\textbf{-3285.35 ($\pm$ 0.00)}&\textbf{-3285.35 ($\pm$ 0.00)}&\textbf{-3285.35 ($\pm$ 0.00)}\\%
& 2.5&\textbf{-3442.45 ($\pm$ 0.00)}&\textbf{-3442.45 ($\pm$ 0.00)}&\textbf{-3442.45 ($\pm$ 0.00)}\\%
& 3.0&\textbf{-3686.19 ($\pm$ 0.00)}&\textbf{-3686.19 ($\pm$ 0.00)}&-3684.57 ($\pm$ 0.00)\\%
& 4.0&\textbf{-4423.08 ($\pm$ 0.00)}&\textbf{-4423.08 ($\pm$ 0.00)}&\textbf{-4423.08 ($\pm$ 0.00)}\\%
\midrule
15 & 1.0&\textbf{-10370.13 ($\pm$ 0.00)}&\textbf{-10370.13 ($\pm$ 0.00)}&\textbf{-10370.13 ($\pm$ 0.00)}\\%
& 2.0&\textbf{-11028.13 ($\pm$ 0.00)}&\textbf{-11028.13 ($\pm$ 0.00)}&\textbf{-11028.13 ($\pm$ 0.00)}\\%
& 2.5&-11552.27 ($\pm$ 0.00)&\textbf{-11552.93 ($\pm$ 1.48)}&-11517.54 ($\pm$ 0.00)\\%
& 3.0&-12383.49 ($\pm$ 2.22)&\textbf{-12385.26 ($\pm$ 1.21)}&-12374.92 ($\pm$ 0.00)\\%
& 4.0&\textbf{-14971.18 ($\pm$ 0.00)}&\textbf{-14971.18 ($\pm$ 0.00)}&\textbf{-14971.18 ($\pm$ 0.00)}\\%
\midrule
20 & 1.0&\textbf{-24652.59 ($\pm$ 0.00)}&\textbf{-24652.59 ($\pm$ 0.00)}&\textbf{-24652.59 ($\pm$ 0.00)}\\%
& 2.0&\textbf{-26304.49 ($\pm$ 0.00)}&\textbf{-26304.49 ($\pm$ 0.00)}&-26250.41 ($\pm$ 0.00)\\%
& 2.5&-27619.80 ($\pm$ 4.21)&\textbf{-27620.45 ($\pm$ 4.83)}&-27544.08 ($\pm$ 0.00)\\%
& 3.0&\textbf{-29645.35 ($\pm$ 1.28)}&-29644.81 ($\pm$ 1.55)&-29610.85 ($\pm$ 0.00)\\%
& 4.0&\textbf{-35565.19 ($\pm$ 0.31)}&-35564.93 ($\pm$ 0.36)&-35565.14 ($\pm$ 0.00)\\%
\midrule
25 & 1.0&\textbf{-48002.16 ($\pm$ 0.00)}&\textbf{-48002.16 ($\pm$ 0.00)}&\textbf{-48002.16 ($\pm$ 0.00)}\\%
& 2.0&-51118.01 ($\pm$ 6.70)&\textbf{-51118.03 ($\pm$ 6.70)}&-50848.18 ($\pm$ 0.00)\\%
& 2.5&\textbf{-53721.84 ($\pm$ 4.68)}&\textbf{-53721.84 ($\pm$ 4.68)}&-53515.69 ($\pm$ 0.00)\\%
& 3.0&\textbf{-57691.34 ($\pm$ 1.18)}&\textbf{-57691.34 ($\pm$ 1.18)}&-57659.92 ($\pm$ 0.00)\\%
& 4.0&-69305.65 ($\pm$ 0.84)&\textbf{-69305.94 ($\pm$ 0.76)}&-69301.26 ($\pm$ 0.00)\\%
\midrule
32 & 1.0&\textbf{-100502.64 ($\pm$ 0.00)}&\textbf{-100502.64 ($\pm$ 0.00)}&\textbf{-100502.64 ($\pm$ 0.00)}\\%
& 2.0&\textbf{-106885.33 ($\pm$ 29.69)}&\textbf{-106885.33 ($\pm$ 29.69)}&-106445.86 ($\pm$ 0.00)\\%
& 2.5&\textbf{-112333.15 ($\pm$ 2.07)}&-112333.02 ($\pm$ 2.30)&-111891.11 ($\pm$ 0.00)\\%
& 3.0&\textbf{-120843.81 ($\pm$ 3.70)}&-120843.71 ($\pm$ 3.61)&-120705.98 ($\pm$ 0.00)\\%
& 4.0&\textbf{-145406.01 ($\pm$ 1.38)}&-145405.72 ($\pm$ 0.77)&-145400.18 ($\pm$ 0.00)\\%
\bottomrule
\end{tabular*}%
\end{table}

\begin{table}[t]
\caption{Comparison on the external field $h_5$ between the mean final function values and standard deviation attained by MultiStart (MS), and the alternating variants of Basin Hopping (BH\_alt) and Population Basin Hopping (PBH\_alt). The lowest mean values are in bold.}
\label{tab:MS_f5}
\begin{tabular*}{\textwidth}{@{\extracolsep\fill}ccccc}
\toprule%
$L$ & $\Delta$ &BH\_alt&PBH\_alt&MS\\%

\midrule
10 & 1.0&\textbf{-3075.65 ($\pm$ 0.00)}&\textbf{-3075.65 ($\pm$ 0.00)}&\textbf{-3075.65 ($\pm$ 0.00)}\\%
& 2.0&\textbf{-3278.45 ($\pm$ 0.00)}&\textbf{-3278.45 ($\pm$ 0.00)}&\textbf{-3278.45 ($\pm$ 0.00)}\\%
& 2.5&\textbf{-3427.16 ($\pm$ 0.61)}&\textbf{-3427.16 ($\pm$ 0.61)}&-3425.29 ($\pm$ 0.00)\\%
& 3.0&\textbf{-3694.43 ($\pm$ 0.14)}&\textbf{-3694.43 ($\pm$ 0.14)}&-3694.37 ($\pm$ 0.00)\\%
& 4.0&\textbf{-4447.34 ($\pm$ 0.00)}&\textbf{-4447.34 ($\pm$ 0.00)}&\textbf{-4447.34 ($\pm$ 0.00)}\\%
\midrule
15 & 1.0&\textbf{-10386.22 ($\pm$ 0.00)}&\textbf{-10386.22 ($\pm$ 0.00)}&\textbf{-10386.22 ($\pm$ 0.00)}\\%
& 2.0&\textbf{-11064.94 ($\pm$ 0.00)}&\textbf{-11064.94 ($\pm$ 0.00)}&\textbf{-11064.94 ($\pm$ 0.00)}\\%
& 2.5&-11586.04 ($\pm$ 0.73)&\textbf{-11587.16 ($\pm$ 1.77)}&-11571.30 ($\pm$ 0.00)\\%
& 3.0&-12411.23 ($\pm$ 2.13)&-12411.79 ($\pm$ 2.30)&\textbf{-12412.88 ($\pm$ 0.00)}\\%
& 4.0&\textbf{-14956.73 ($\pm$ 0.00)}&\textbf{-14956.73 ($\pm$ 0.00)}&\textbf{-14956.73 ($\pm$ 0.00)}\\%
\midrule
20 & 1.0&\textbf{-24554.84 ($\pm$ 0.00)}&\textbf{-24554.84 ($\pm$ 0.00)}&\textbf{-24554.84 ($\pm$ 0.00)}\\%
& 2.0&\textbf{-26050.73 ($\pm$ 2.10)}&\textbf{-26050.73 ($\pm$ 2.10)}&-25922.62 ($\pm$ 0.00)\\%
& 2.5&\textbf{-27410.76 ($\pm$ 6.07)}&\textbf{-27410.76 ($\pm$ 6.07)}&-27353.30 ($\pm$ 5.86)\\%
& 3.0&-29534.36 ($\pm$ 1.18)&\textbf{-29535.22 ($\pm$ 1.61)}&-29519.80 ($\pm$ 0.00)\\%
& 4.0&-35554.21 ($\pm$ 0.19)&\textbf{-35554.34 ($\pm$ 0.13)}&-35552.17 ($\pm$ 0.00)\\%
\midrule
25 & 1.0&\textbf{-47966.45 ($\pm$ 0.00)}&\textbf{-47966.45 ($\pm$ 0.00)}&\textbf{-47966.45 ($\pm$ 0.00)}\\%
& 2.0&\textbf{-50852.84 ($\pm$ 38.39)}&\textbf{-50852.84 ($\pm$ 38.40)}&-50795.00 ($\pm$ 0.00)\\%
& 2.5&\textbf{-53701.34 ($\pm$ 4.83)}&\textbf{-53701.34 ($\pm$ 4.83)}&-53579.16 ($\pm$ 0.00)\\%
& 3.0&-57895.50 ($\pm$ 5.12)&\textbf{-57896.02 ($\pm$ 3.59)}&-57856.31 ($\pm$ 0.00)\\%
& 4.0&\textbf{-69657.23 ($\pm$ 0.47)}&-69656.94 ($\pm$ 0.38)&-69656.44 ($\pm$ 0.00)\\%
\midrule
32 & 1.0&\textbf{-100457.22 ($\pm$ 0.00)}&\textbf{-100457.22 ($\pm$ 0.00)}&\textbf{-100457.22 ($\pm$ 0.00)}\\%
& 2.0&\textbf{-106694.07 ($\pm$ 37.32)}&-106693.52 ($\pm$ 36.66)&-106341.45 ($\pm$ 0.00)\\%
& 2.5&\textbf{-112509.65 ($\pm$ 8.28)}&\textbf{-112509.65 ($\pm$ 8.28)}&-112205.00 ($\pm$ 0.00)\\%
& 3.0&\textbf{-121119.65 ($\pm$ 6.56)}&\textbf{-121119.65 ($\pm$ 6.56)}&-120990.72 ($\pm$ 0.00)\\%
& 4.0&\textbf{-145810.92 ($\pm$ 0.84)}&-145809.87 ($\pm$ 1.60)&-145799.74 ($\pm$ 0.00)\\%
\bottomrule
\end{tabular*}%
\end{table}

\subsection{Comparison with Simulated Annealing, Parallel Tempering, and Differential Evolution}
Finally, we compare the best-performing methods identified in the previous experiments with some state-of-the-art heuristic algorithms for highly irregular problems: DE, a widely used genetic optimization algorithm; two optimization methods commonly adopted in computational physics, namely SA and PT. The purpose of this comparison is mainly to assess the competitiveness of the proposed approaches against established techniques commonly used for spin systems and related energy minimization problems.
Figure \ref{fig:genetici} compares these three algorithms with the alternating variant PBH\_alt. The comparison is carried out with PBH\_alt, rather than BH\_alt, to ensure methodological consistency with the competing population-based methods. This choice does not affect the conclusions, since the previous experiments showed that the two algorithms exhibit very similar performance on the considered benchmark instances. As shown in Figure \ref{fig:genetici_cumulative}, PBH\_alt outperforms the other algorithms in terms of the cumulative distribution of the final objective values, followed by DE, PT, and SA. Furthermore, Figure \ref{fig:genetici_flow} shows that PBH\_alt also achieves the best performance in terms of the percentage of problems for which it attains a better objective value within a given number of iterations, with DE ranking second, while PT and SA exhibit the poorest performance.
The superiority of PBH\_alt is further confirmed by Tables \ref{tab:genetici_f2}--\ref{tab:genetici_f4}, which report the mean final objective values and standard deviation for the external fields $h_2$ and $h_4$, respectively, as representative examples (the results for the other three external fields are analogous and are therefore omitted). Interestingly, DE consistently outperforms both PT and SA, indicating that the genetic approach is more effective than the optimization methods traditionally employed in computational physics for this problem. 

\begin{figure}
    \subfloat[Cumulative distribution of the relative gap of the final objective value with respect to the best value found by any of the solvers.]{\includegraphics[width=0.47\textwidth]{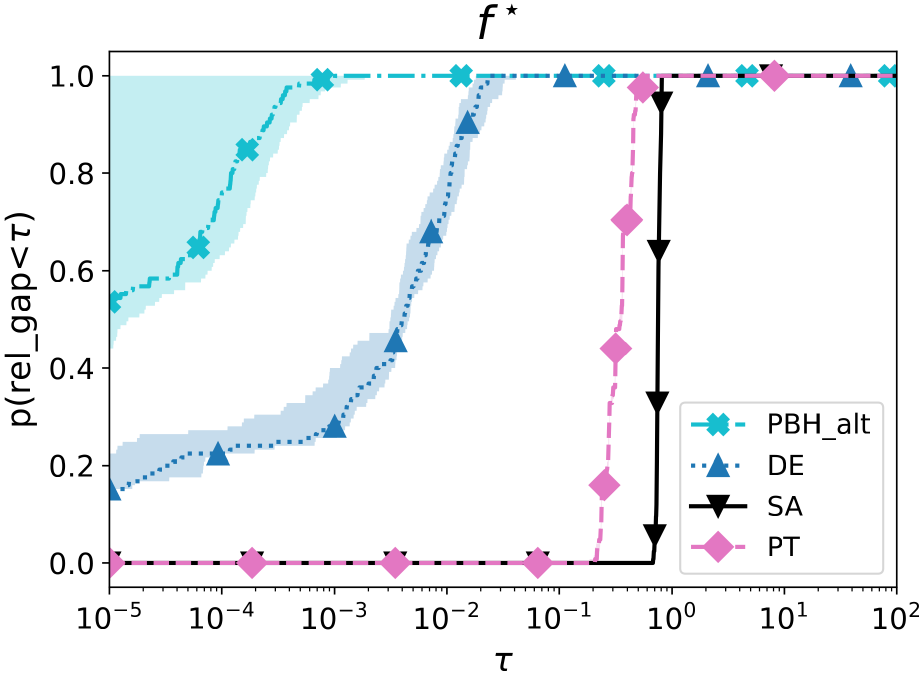} \label{fig:genetici_cumulative}}
    \hfill
    \subfloat[Percentage of problems for which an algorithm attains a better objective value within a given number of iterations.]{\includegraphics[width=0.435\textwidth]{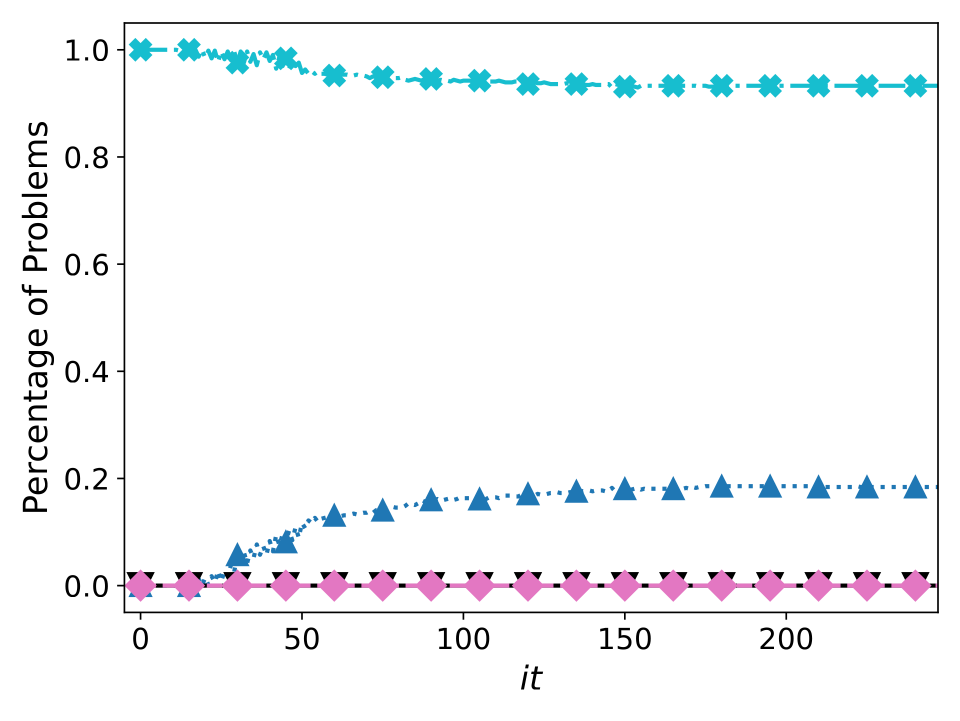}
    \label{fig:genetici_flow}}
    \caption{Results of the comparison between the alternating variant of Population Basin Hopping (PBH\_alt) and the three genetic algorithms: Differential Evolution (DE), Simulated Annealing (SA) and Parallel Tempering (PT).}
    \label{fig:genetici}
\end{figure}

\begin{table}[t]
\caption{Comparison on the external field $h_2$ between the mean final function values and standard deviation attained by the alternating variant of Population Basin Hopping (PBH\_alt) and the three genetic algorithms: Differential Evolution (DE), Simulated Annealing (SA) and Parallel Tempering (PT). The lowest mean values are in bold, while the second-lowest mean values are underlined.}
\label{tab:genetici_f2}
\begin{tabular*}{\textwidth}{@{\extracolsep\fill}cccccc}
\toprule%
$L$ & $\Delta$ &PBH\_alt&DE&SA&PT\\%
\midrule
10 & 1.0&\textbf{-3117.45 ($\pm$ 0.00)}&\textbf{-3117.45 ($\pm$ 0.00)}&-662.91 ($\pm$ 17.54)&\underline{-1751.15 ($\pm$ 13.96)}\\%
& 2.0&\textbf{-3341.93 ($\pm$ 0.00)}&\textbf{-3341.93 ($\pm$ 0.00)}&-840.19 ($\pm$ 6.25)&\underline{-2117.33 ($\pm$ 9.59)}\\%
& 2.5&\textbf{-3499.03 ($\pm$ 0.00)}&\textbf{-3499.03 ($\pm$ 0.00)}&-974.24 ($\pm$ 22.95)&\underline{-2399.05 ($\pm$ 11.70)}\\%
& 3.0&\textbf{-3695.91 ($\pm$ 0.00)}&\underline{-3691.95 ($\pm$ 1.08)}&-1110.87 ($\pm$ 12.15)&-2708.42 ($\pm$ 14.50)\\%
& 4.0&\textbf{-4379.95 ($\pm$ 0.00)}&\underline{-4379.44 ($\pm$ 0.47)}&-1424.23 ($\pm$ 14.46)&-3430.14 ($\pm$ 17.13)\\%
\midrule
15 & 1.0&\textbf{-10414.48 ($\pm$ 0.00)}&\textbf{-10414.48 ($\pm$ 0.00)}&-2094.51 ($\pm$ 7.16)&\underline{-4691.25 ($\pm$ 39.35)}\\%
& 2.0&\textbf{-11123.44 ($\pm$ 0.00)}&\underline{-11016.77 ($\pm$ 4.73)}&-2599.07 ($\pm$ 27.36)&-5813.77 ($\pm$ 39.70)\\%
& 2.5&\textbf{-11658.33 ($\pm$ 0.00)}&\underline{-11533.37 ($\pm$ 20.01)}&-2939.09 ($\pm$ 31.38)&-6634.46 ($\pm$ 39.55)\\%
& 3.0&\textbf{-12416.90 ($\pm$ 0.98)}&\underline{-12399.70 ($\pm$ 3.66)}&-3307.58 ($\pm$ 40.82)&-7515.35 ($\pm$ 48.53)\\%
& 4.0&\textbf{-14915.69 ($\pm$ 0.00)}&\underline{-14915.47 ($\pm$ 0.02)}&-4118.71 ($\pm$ 58.28)&-9562.96 ($\pm$ 76.02)\\%
\midrule
20 & 1.0&\textbf{-24634.42 ($\pm$ 0.00)}&\textbf{-24634.42 ($\pm$ 0.00)}&-4995.08 ($\pm$ 44.36)&\underline{-13645.07 ($\pm$ 94.20)}\\%
& 2.0&\textbf{-26241.02 ($\pm$ 0.00)}&\underline{-26060.64 ($\pm$ 58.36)}&-6126.20 ($\pm$ 70.08)&-16627.01 ($\pm$ 69.76)\\%
& 2.5&\textbf{-27475.04 ($\pm$ 0.44)}&\underline{-27313.71 ($\pm$ 28.78)}&-6866.72 ($\pm$ 78.51)&-18769.23 ($\pm$ 62.56)\\%
& 3.0&\textbf{-29316.82 ($\pm$ 1.20)}&\underline{-29253.44 ($\pm$ 5.55)}&-7676.06 ($\pm$ 106.59)&-21328.61 ($\pm$ 42.47)\\%
& 4.0&\textbf{-35241.76 ($\pm$ 0.18)}&\underline{-35240.42 ($\pm$ 0.70)}&-9482.22 ($\pm$ 109.27)&-27066.22 ($\pm$ 105.11)\\%
\midrule
25 & 1.0&\textbf{-47988.51 ($\pm$ 0.00)}&\underline{-47482.39 ($\pm$ 194.35)}&-9586.37 ($\pm$ 13.86)&-26394.13 ($\pm$ 83.62)\\%
& 2.0&\textbf{-51128.90 ($\pm$ 5.73)}&\underline{-50528.95 ($\pm$ 78.74)}&-11872.44 ($\pm$ 43.88)&-32399.20 ($\pm$ 89.83)\\%
& 2.5&\textbf{-53725.01 ($\pm$ 3.33)}&\underline{-53409.66 ($\pm$ 94.92)}&-13365.09 ($\pm$ 72.37)&-36699.24 ($\pm$ 50.50)\\%
& 3.0&\textbf{-57682.57 ($\pm$ 4.03)}&\underline{-57395.83 ($\pm$ 23.34)}&-15007.92 ($\pm$ 99.72)&-41890.56 ($\pm$ 131.54)\\%
& 4.0&\textbf{-69459.08 ($\pm$ 0.54)}&\underline{-69371.27 ($\pm$ 94.43)}&-18620.86 ($\pm$ 153.22)&-53010.03 ($\pm$ 116.86)\\%
\midrule
32 & 1.0&\textbf{-100412.29 ($\pm$ 0.00)}&\underline{-98826.05 ($\pm$ 330.14)}&-19841.79 ($\pm$ 135.55)&-55306.58 ($\pm$ 100.44)\\%
& 2.0&\textbf{-106770.66 ($\pm$ 8.28)}&\underline{-105334.03 ($\pm$ 162.61)}&-24353.73 ($\pm$ 159.39)&-67742.91 ($\pm$ 146.66)\\%
& 2.5&\textbf{-112336.50 ($\pm$ 20.03)}&\underline{-111299.47 ($\pm$ 53.96)}&-27327.34 ($\pm$ 192.10)&-76891.30 ($\pm$ 140.55)\\%
& 3.0&\textbf{-120783.09 ($\pm$ 5.05)}&\underline{-120266.12 ($\pm$ 50.08)}&-30672.47 ($\pm$ 232.58)&-87285.44 ($\pm$ 161.65)\\%
& 4.0&\textbf{-145585.07 ($\pm$ 0.57)}&\underline{-145108.48 ($\pm$ 20.44)}&-38056.27 ($\pm$ 324.95)&-110818.76 ($\pm$ 126.74)\\%
\bottomrule
\end{tabular*}%
\end{table}

\begin{table}[t]
\caption{Comparison on the external field $h_4$ between the mean final function values and standard deviation attained by the alternating variant of Population Basin Hopping (PBH\_alt) and the three genetic algorithms: Differential Evolution (DE), Simulated Annealing (SA) and Parallel Tempering (PT). The lowest mean values are in bold, while the second-lowest mean values are underlined.}
\label{tab:genetici_f4}
\begin{tabular*}{\textwidth}{@{\extracolsep\fill}cccccc}
\toprule%
$L$ & $\Delta$ &PBH\_alt&DE&SA&PT\\%
\midrule
10 & 1.0&\textbf{-3085.65 ($\pm$ 0.00)}&\textbf{-3085.65 ($\pm$ 0.00)}&-660.00 ($\pm$ 16.36)&\underline{-1753.96 ($\pm$ 12.42)}\\%
& 2.0&\textbf{-3292.50 ($\pm$ 0.00)}&\textbf{-3292.50 ($\pm$ 0.00)}&-819.42 ($\pm$ 31.23)&\underline{-2117.46 ($\pm$ 30.39)}\\%
& 2.5&\textbf{-3444.31 ($\pm$ 0.49)}&\underline{-3437.24 ($\pm$ 0.51)}&-933.37 ($\pm$ 40.66)&-2407.56 ($\pm$ 36.24)\\%
& 3.0&\textbf{-3656.43 ($\pm$ 0.00)}&\underline{-3654.42 ($\pm$ 1.36)}&-1059.02 ($\pm$ 38.61)&-2690.35 ($\pm$ 15.04)\\%
& 4.0&\textbf{-4362.08 ($\pm$ 0.00)}&\underline{-4361.99 ($\pm$ 0.16)}&-1328.37 ($\pm$ 43.57)&-3415.41 ($\pm$ 23.57)\\%
\midrule
15 & 1.0&\textbf{-10405.97 ($\pm$ 0.00)}&\textbf{-10405.97 ($\pm$ 0.00)}&-2079.55 ($\pm$ 8.03)&\underline{-4760.01 ($\pm$ 26.79)}\\%
& 2.0&\textbf{-11103.63 ($\pm$ 0.00)}&\underline{-10933.11 ($\pm$ 24.16)}&-2571.51 ($\pm$ 33.03)&-5819.37 ($\pm$ 49.23)\\%
& 2.5&\textbf{-11627.46 ($\pm$ 3.64)}&\underline{-11567.53 ($\pm$ 8.02)}&-2902.02 ($\pm$ 41.38)&-6611.39 ($\pm$ 36.62)\\%
& 3.0&\textbf{-12489.34 ($\pm$ 0.40)}&\underline{-12472.37 ($\pm$ 2.07)}&-3266.55 ($\pm$ 61.43)&-7520.88 ($\pm$ 25.72)\\%
& 4.0&\textbf{-14938.10 ($\pm$ 0.00)}&\underline{-14937.96 ($\pm$ 0.01)}&-4043.25 ($\pm$ 92.25)&-9500.58 ($\pm$ 43.22)\\%
\midrule
20 & 1.0&\textbf{-24612.20 ($\pm$ 0.00)}&\underline{-24612.19 ($\pm$ 0.01)}&-4922.84 ($\pm$ 13.99)&-13666.45 ($\pm$ 83.83)\\%
& 2.0&\textbf{-26233.66 ($\pm$ 0.00)}&\underline{-25946.30 ($\pm$ 67.52)}&-6021.10 ($\pm$ 24.53)&-16632.73 ($\pm$ 36.99)\\%
& 2.5&\textbf{-27467.70 ($\pm$ 6.21)}&\underline{-27362.91 ($\pm$ 9.99)}&-6791.27 ($\pm$ 19.04)&-18822.61 ($\pm$ 21.53)\\%
& 3.0&\textbf{-29540.51 ($\pm$ 2.04)}&\underline{-29494.22 ($\pm$ 5.66)}&-7645.28 ($\pm$ 38.86)&-21406.97 ($\pm$ 62.77)\\%
& 4.0&\textbf{-35503.55 ($\pm$ 0.15)}&\underline{-35502.05 ($\pm$ 0.83)}&-9517.61 ($\pm$ 54.64)&-27189.13 ($\pm$ 44.27)\\%
\midrule
25 & 1.0&\textbf{-47893.15 ($\pm$ 0.00)}&\underline{-47436.45 ($\pm$ 524.58)}&-9420.91 ($\pm$ 10.13)&-26314.91 ($\pm$ 65.69)\\%
& 2.0&\textbf{-50938.75 ($\pm$ 23.98)}&\underline{-50477.08 ($\pm$ 62.14)}&-11511.44 ($\pm$ 27.24)&-32207.92 ($\pm$ 57.78)\\%
& 2.5&\textbf{-53699.59 ($\pm$ 3.15)}&\underline{-53354.57 ($\pm$ 75.47)}&-12941.64 ($\pm$ 43.38)&-36558.04 ($\pm$ 70.05)\\%
& 3.0&\textbf{-57656.85 ($\pm$ 4.76)}&\underline{-57447.27 ($\pm$ 10.68)}&-14558.37 ($\pm$ 49.76)&-41545.10 ($\pm$ 52.37)\\%
& 4.0&\textbf{-69241.69 ($\pm$ 0.30)}&\underline{-69071.09 ($\pm$ 61.56)}&-18090.55 ($\pm$ 120.56)&-52863.67 ($\pm$ 91.55)\\%
\midrule
32 & 1.0&\textbf{-100633.67 ($\pm$ 0.00)}&\underline{-99571.62 ($\pm$ 300.73)}&-19751.73 ($\pm$ 37.08)&-55420.43 ($\pm$ 124.41)\\%
& 2.0&\textbf{-107078.84 ($\pm$ 21.86)}&\underline{-105509.15 ($\pm$ 145.08)}&-24186.73 ($\pm$ 80.58)&-67584.37 ($\pm$ 166.83)\\%
& 2.5&\textbf{-112426.86 ($\pm$ 9.38)}&\underline{-111624.85 ($\pm$ 66.71)}&-27103.14 ($\pm$ 111.94)&-76603.67 ($\pm$ 75.79)\\%
& 3.0&\textbf{-120812.62 ($\pm$ 5.12)}&\underline{-120263.61 ($\pm$ 33.48)}&-30362.13 ($\pm$ 155.54)&-87044.37 ($\pm$ 103.93)\\%
& 4.0&\textbf{-145408.68 ($\pm$ 0.69)}&\underline{-144923.93 ($\pm$ 20.54)}&-37654.13 ($\pm$ 242.10)&-110563.64 ($\pm$ 141.47)\\%
\bottomrule
\end{tabular*}%
\end{table}

\section{Conclusions}
\label{Sec:conclusions}

This study investigates a global optimization problem of continuous type, associated with the non-convex energy landscape of a classical statistical physics system, namely the random field $XY$ model. Through this work, we demonstrated the usefulness of optimization techniques for the study of physical systems. We have introduced a Riemannian manifold formulation of the problem, which allows for more efficient and numerically stable computations compared to the standard unconstrained angular-variable representation.
Building upon this formulation, we proposed problem-specific perturbation strategies within the Basin Hopping framework and extended them to a Population Basin Hopping scheme. The computational study demonstrated that exploiting the geometric and physical structure of the problem substantially improves the exploration of the energy landscape. In particular, the combination of the alternating perturbation strategy with the population-based framework consistently achieved the best overall performance, outperforming not only the other proposed variants but also the optimization methods traditionally adopted for the random field $XY$ model.
Overall, this work demonstrated the effectiveness of combining modern global optimization algorithms with Riemannian optimization techniques in a physical context, providing a useful framework for exploring complex energy landscapes in disordered continuous-spin systems.
From both optimization and physical perspectives, it is of interest to extend the present study to larger system sizes $(L > 100)$, for which the energy landscape is expected to exhibit an extremely large number of local optima. Such large-scale instances are particularly relevant in statistical physics, as they enable the investigation of \textit{true} thermodynamic behavior and the mitigation of finite-size effects~\cite{XYmodel,PhysRevB.91.134203}. The manifold basin-hopping variants introduced in this work provide a solid foundation for tackling this regime, motivating further algorithmic refinements to efficiently scale the exploration of such highly multimodal landscapes.
Another direction for future research is the extension of the proposed framework to other classes of disordered spin systems and constrained optimization problems on manifolds, such as vector spin glasses and random field $\mathcal{O}(n)$ models.


\backmatter

\section*{Declarations}

\bmhead{Acknowledgements}

The authors are thankful to the anonymous referee of this manuscript for the constructive comments that helped to improve the quality of the work.

\bmhead{Funding}

EM and RA have been supported by funding from the 2021 first FIS
(Fondo Italiano per la Scienza) funding scheme (FIS783 - SMaC -
Statistical Mechanics and Complexity) from
Italian MUR (Ministry of University and Research).

\bmhead{Competing interests}

The authors have no competing interests to declare that are relevant to the content of this article.

\bmhead{Data Availability Statement}

The codes used for the experiments, along with  experimental outputs and the external fields, are available in the following GitHub repository: \href{https://github.com/LorenzoCiarpa/RFXYBH-Riemann}{\tt https://github.com/LorenzoCiarpa/RFXYBH-Riemann}.

\appendix
\section{Local solvers comparison}
\label{sec:app_local}

In this appendix, we provide the results regarding the choice of the local solver. We compare two local algorithms adapted for Riemannian optimization: Riemannian Conjugate Gradient (RCG) and Riemannian Trust Region (RTR) \cite{rtr}. Analyses are performed for different values of $\Delta \in \{0.001, 0.1, 0.5, 1.0, 1.5, 2.0, 2.5, 3.0, 4.0, 5.0\}$ and for system sizes $L\in\{10, 20, 32\}$, considering a three-dimensional lattice ($d=3$). To better explore the energy landscape, each algorithm is run $200$ times from different random initializations while keeping the external field fixed. For consistency with the computational study presented in Section~\ref{sec:results}, the external field realizations $h_1$ and $h_2$ are adopted. 

\par\medskip\noindent
For RCG and RTR we used the \texttt{conjugate\_gradient} and \texttt{trust\_regions} routines provided by Python's package \texttt{pymanopt} \cite{pymanopt}, with the default parameters. We set a tolerance of $10^{-6}$ for the norm of the Riemannian gradient and a maximum number of iterations equal to $10,\!000$.
The goal of this comparison is to identify the \textit{best} local solver in terms of both the number of distinct local minima found and the computational time required. As shown in Table \ref{tab:time_results_local}, the RTR method achieves the lowest computational time in most of the considered cases across all lattice sizes and values of $\Delta$. Moreover, the two solvers are equivalent in terms of the solution quality, i.e., they almost reach the same minimum value for $f$. Specifically, for large values of $\Delta$, all solvers usually reach the same solution $\theta^*$ at each run. For small values of $\Delta$, they converge to different points; since the value of $\Delta$ is not sufficiently close to zero, they sometimes reach local minima, although in most cases they still converge to the putative ground state. Instead, for intermediate values, they usually obtain different solutions with different function values, but the lowest value of $f$ is almost reached by all solvers. As an example, solutions obtained for $L=32$ and $\Delta\in\{0.001,3.0,5.0\}$ are shown in Figure \ref{fig:loc_L32_all}. These results corroborate the theoretical analysis discussed in Section \ref{sec:problem}.
Since RTR, being a second-order method, typically finds more accurate local minima than RCG and showed slightly better performance in our experiments, it was selected as the local solver in BH, PBH and MS.

\begin{table}[t]
\caption{Mean execution time of each algorithm (RCG and RTR) as a function of the parameter $\Delta$, for different system sizes $L$ and external fields $h_1$ and $h_2$. All values are expressed in seconds and rounded to three decimal places. The best time is highlighted in bold.}
\label{tab:time_results_local}
\begin{tabular}{@{}l rrrr rrrr rrrr @{}}
\toprule

& \multicolumn{4}{c}{$L=10$} 
& \multicolumn{4}{c}{$L=20$}
& \multicolumn{4}{c}{$L=32$} \\

\cmidrule(lr){2-5} \cmidrule(lr){6-9} \cmidrule(lr){10-13}

$\Delta$ 
& \multicolumn{2}{c}{$h_1$} & \multicolumn{2}{c}{$h_2$}
& \multicolumn{2}{c}{$h_1$} & \multicolumn{2}{c}{$h_2$}
& \multicolumn{2}{c}{$h_1$} & \multicolumn{2}{c}{$h_2$} \\

\cmidrule(lr){2-3} \cmidrule(lr){4-5}
\cmidrule(lr){6-7} \cmidrule(lr){8-9}
\cmidrule(lr){10-11} \cmidrule(lr){12-13}

& RCG & RTR & RCG & RTR 
& RCG & RTR & RCG & RTR 
& RCG & RTR & RCG & RTR \\

\midrule

0.001 & 0.913 & \textbf{0.105} & 0.682 & \textbf{0.108} & 24.536 & \textbf{1.340} & 19.086 & \textbf{1.254} & 99.873 & \textbf{17.994} & 114.48 & \textbf{10.374} \\
0.1   & 0.307 & \textbf{0.100} & 0.224 & \textbf{0.095} & 3.522 & \textbf{1.263} & 3.062 & \textbf{1.112} & 27.681 & \textbf{9.427} & 41.564 & \textbf{7.627} \\
0.5   & 0.212 & \textbf{0.100} & 0.160 & \textbf{0.072} & 1.971 & \textbf{1.321} & 1.912 & \textbf{0.926} & 15.494 & \textbf{10.158} & 19.473 & \textbf{6.175} \\
1.0   & 0.174 & \textbf{0.097} & 0.112 & \textbf{0.078} & 1.456 & \textbf{1.154} & \textbf{1.024} & 1.096 & 11.359 & \textbf{10.856} & 13.537 & \textbf{7.715} \\
1.5   & 0.182 & \textbf{0.115} & 0.105 & \textbf{0.079} & 1.239 & \textbf{1.179} & \textbf{1.026} & 1.268 & 8.209 & \textbf{4.328} & 7.813 & \textbf{6.754} \\
2.0   & 0.139 & \textbf{0.129} & 0.097 & \textbf{0.073} & 0.826 & \textbf{0.785} & \textbf{0.921} & 1.396 & 6.173 & \textbf{3.851} & \textbf{5.721} & 6.192 \\
2.5   & 0.187 & \textbf{0.104} & 0.118 & \textbf{0.101} & \textbf{0.722} & 0.801 & \textbf{0.755} & 1.183 & 5.177 & \textbf{3.505} & \textbf{4.675} & 4.715 \\
3.0   & 0.201 & \textbf{0.116} & 0.112 & \textbf{0.109} & \textbf{0.659} & 0.734 & \textbf{0.670} & 0.969 & 3.711 & \textbf{3.003} & \textbf{4.344} & 4.619 \\
4.0   & 0.125 & \textbf{0.069} & 0.114 & \textbf{0.088} & \textbf{0.472} & 0.486 & \textbf{0.517} & 0.674 & 3.418 & \textbf{2.105} & 3.648 & \textbf{2.422} \\
5.0   & 0.099 & \textbf{0.051} & 0.076 & \textbf{0.071} & \textbf{0.386} & 0.403 & 0.411 & \textbf{0.387} & 1.329 & \textbf{0.937} & 2.320 & \textbf{2.018} \\

\botrule
\end{tabular}
\end{table}

\begin{figure}[h!]
    \centering
    \begin{subfigure}[t]{\linewidth}
         \centering
          \includegraphics[width=0.48\linewidth]{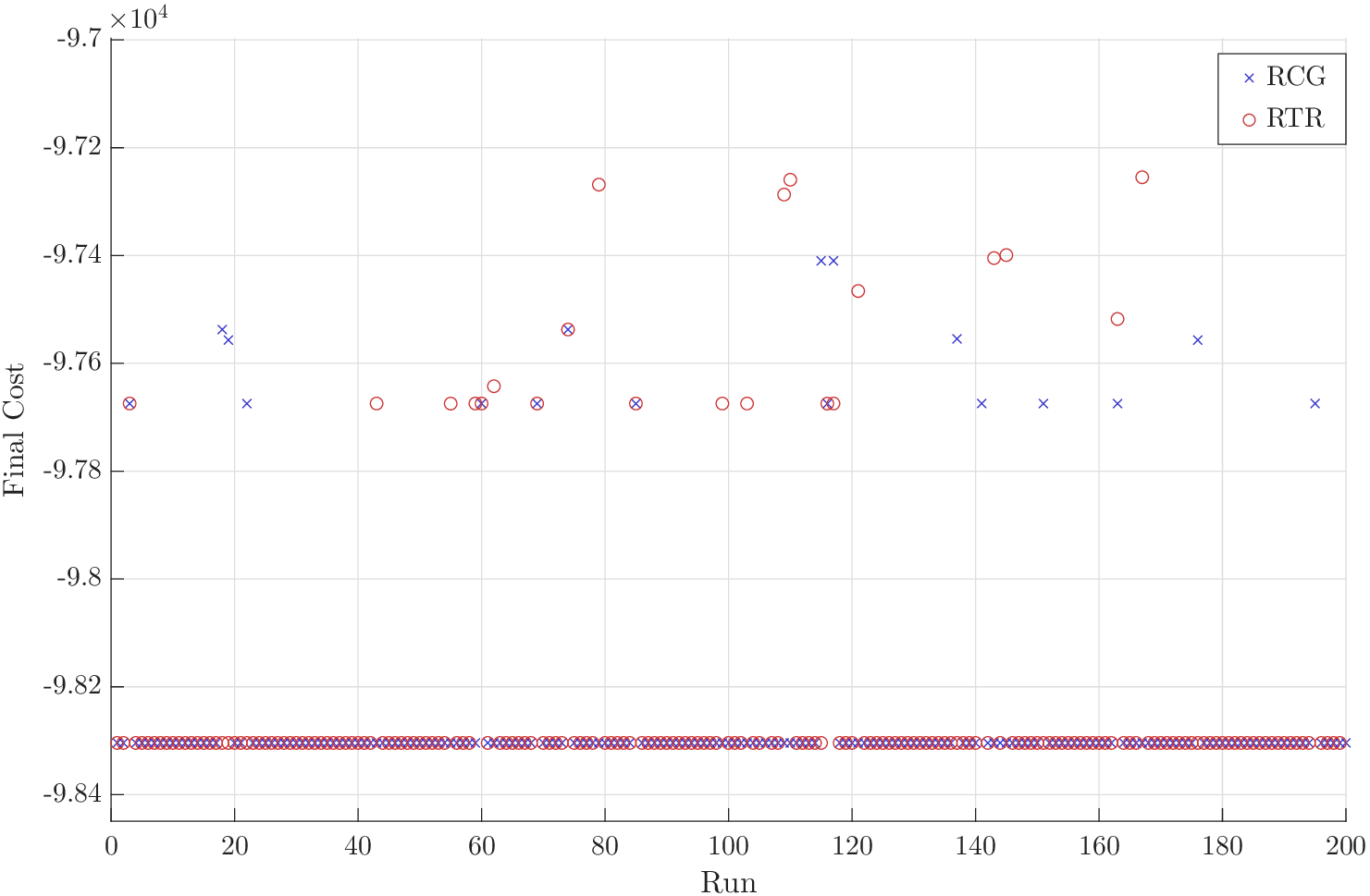}
          \includegraphics[width=0.48\linewidth]{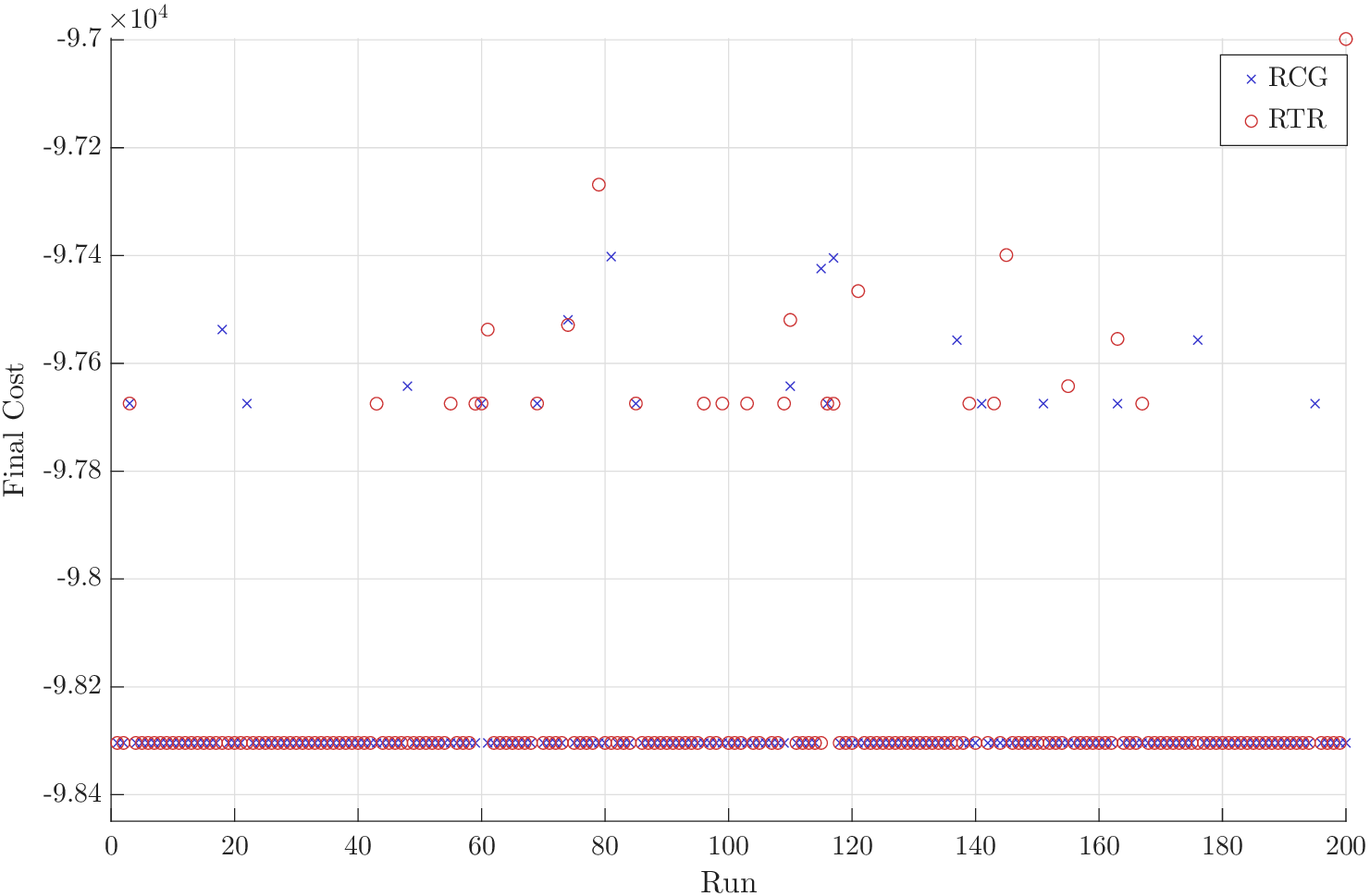}
         \caption{$\Delta=0.001$}
     \end{subfigure}
     \begin{subfigure}[t]{\linewidth}
         \centering
          \includegraphics[width=0.48\linewidth]{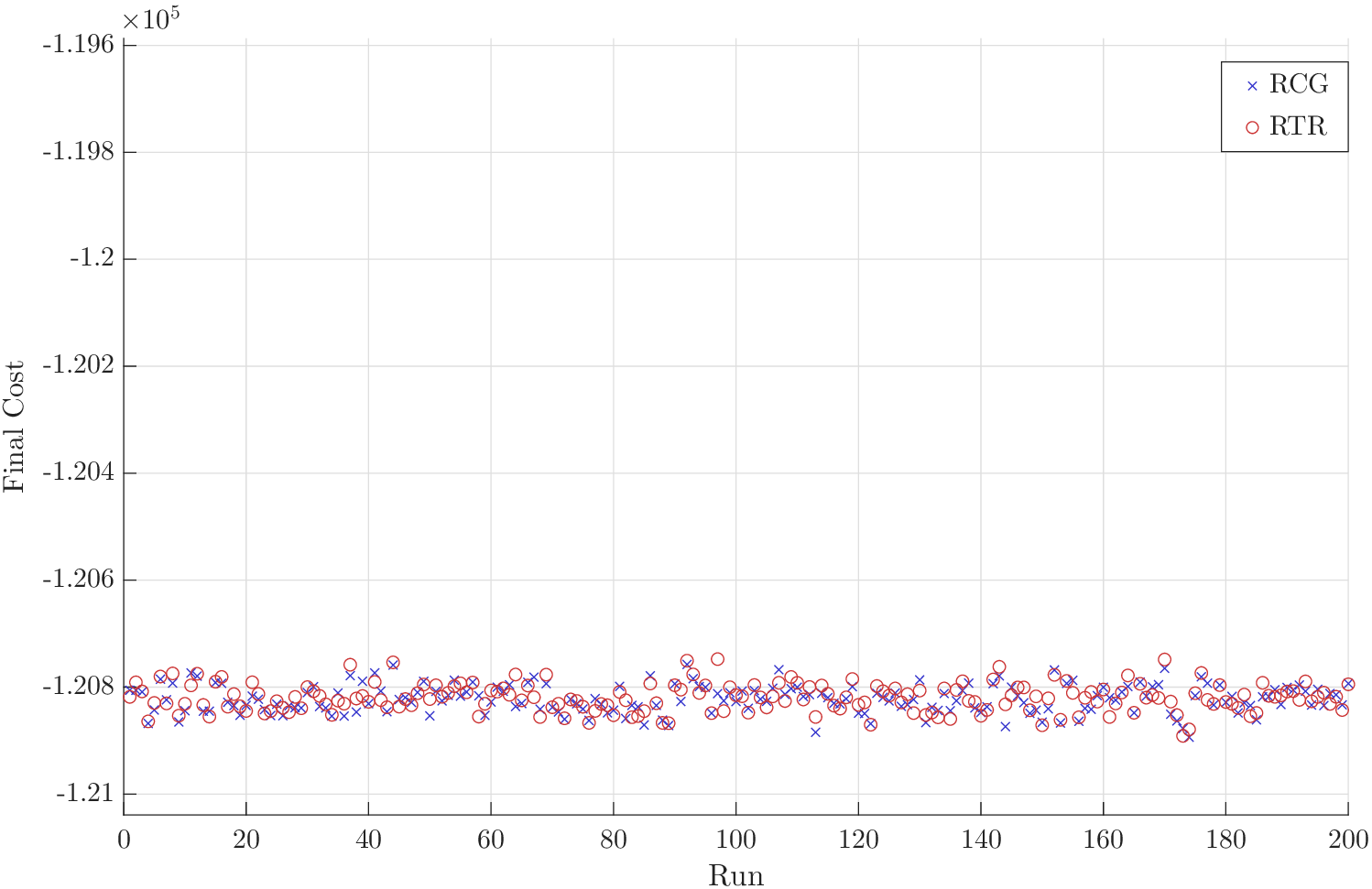}
          \includegraphics[width=0.48\linewidth]{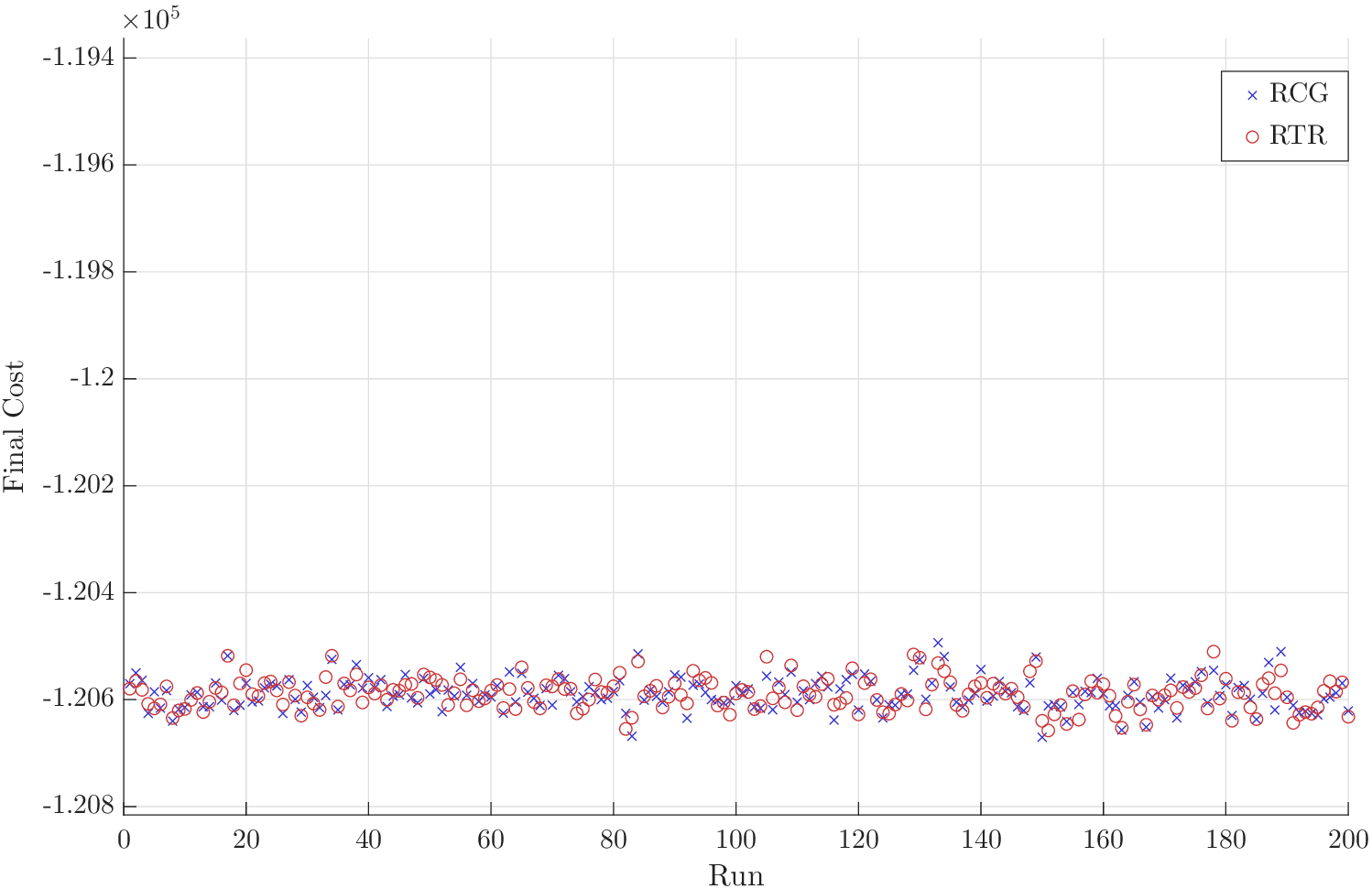}
         \caption{$\Delta=3.0$}
     \end{subfigure}
     \begin{subfigure}[t]{\linewidth}
         \centering
          \includegraphics[width=0.48\linewidth]{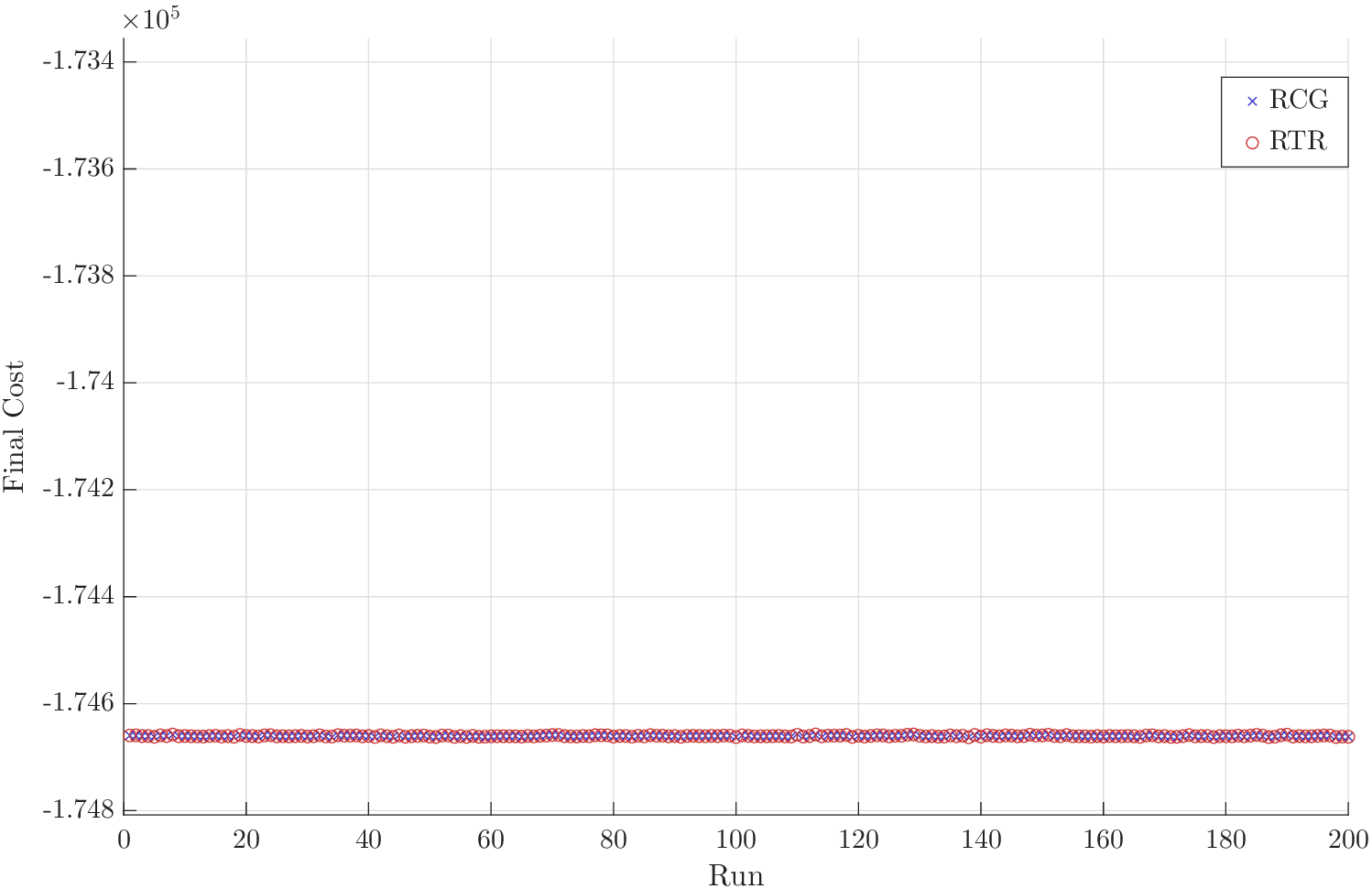}
          \includegraphics[width=0.48\linewidth]{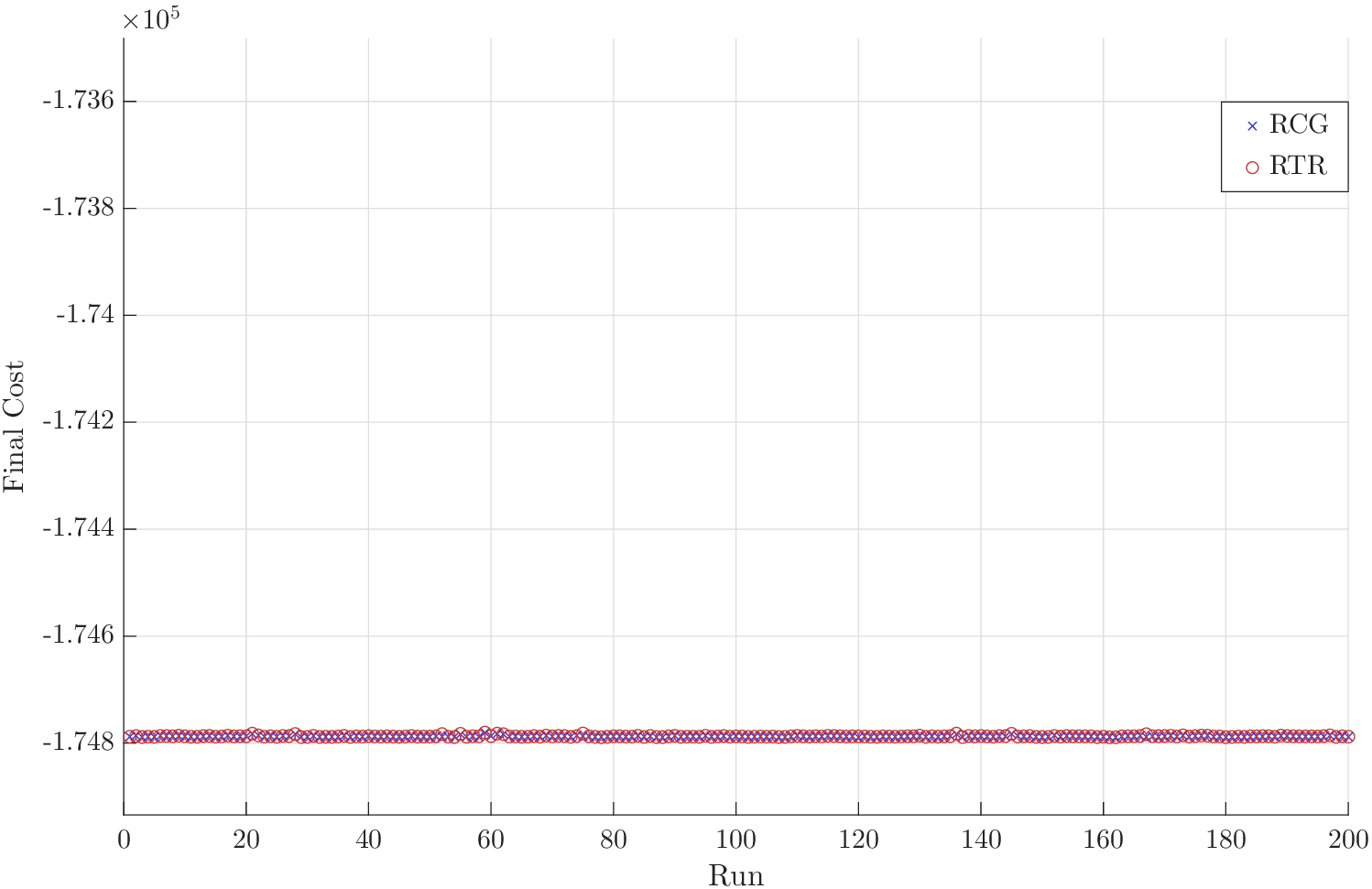}
         \caption{$\Delta=5.0$}
     \end{subfigure}
    \caption{Comparison of local results for different values of $\Delta$ with system size $L=32$, showing the objective function values obtained over 200 independent runs. The results for $h_1$ are shown on the left, while those for $h_2$ are shown on the right.}
    \label{fig:loc_L32_all}
    
\end{figure}

\bibliography{bibliography}

\end{document}